\documentclass[12pt,leqno]{article}
\usepackage{amsmath} 
\usepackage{amssymb} 
\usepackage{amsxtra} 
\usepackage{tabularx} 
\usepackage{theorem}
\usepackage[dvips]{graphicx}
\usepackage[arrow,matrix,tips,frame,color,line,poly]{xy}
\DeclareMathOperator{\Sympl}{Sp}

\newtheorem{theorem}[equation]{Theorem}
\newtheorem{corollary}[equation]{Corollary}
\newtheorem{definition}[equation]{Definition}
\newtheorem{lemma}[equation]{Lemma}

\newtheorem{proposition}[equation]{Proposition}

{\theorembodyfont{\rmfamily}
\newtheorem{theoremplain}[equation]{Theorem}
\newtheorem{remarkplain}[equation]{Remark}

\newtheorem{exampleplain}[equation]{Example}

}

\newcommand{\qed}{\hfill $\square$ \medskip}
\newenvironment{proof}[1][Proof]{\noindent\textbf{#1.} }{\qed}
\newcommand\exact[3]{1\rightarrow #1\rightarrow #2\rightarrow #3\rightarrow1}
\newcommand{\Aut}{\text{Aut}}
\newcommand{\diag}{\text{diag}}
\newcommand{\gr}{\text{gr}}
\newcommand{\Out}{\text{Out}}
\newcommand{\Int}{\text{Int}}
\renewcommand{\int}{\text{int}}
\newcommand{\Hom}{\text{Hom}}
\newcommand{\Id}{\text{Id}}
\newcommand{\Homadm}{\text{Hom}_{\text{adm}}}
\newcommand{\Ind}{\text{Ind}}
\newcommand{\End}{\text{End}}

\newcommand{\Ad}{\text{Ad}}
\newcommand{\Lie}{\text{Lie}}
\newcommand{\Stab}{\text{Stab}}
\newcommand{\Norm}{\text{Norm}}
\newcommand{\Cent}{\text{Cent}}
\newcommand{\X}{\mathcal X}

\newcommand{\caZ}{\mathcal Z}

\newcommand{\I}{\mathcal I}
\renewcommand{\O}{\mathcal O}

\newcommand{\caP}{\mathcal P}
\newcommand{\caH}{\mathcal H}
\newcommand{\Preg}{P_{\text{reg}}}
\newcommand{\chPreg}{\ch P_{\text{reg}}}

\newcommand{\caC}{\mathcal C}
\newcommand{\tX}{\widetilde{\mathcal X}}

\newcommand{\F}{\mathbb F}
\newcommand{\R}{\mathbb R}
\newcommand{\C}{\mathbb C}
\newcommand{\Z}{\mathbb Z}
\newcommand{\Ztwo}{\mathbb Z/2\Z}

\newcommand{\Q}{\mathbb Q}
\renewcommand{\F}{\mathbb F}

\newcommand{\caD}{\mathcal D}

\newcommand{\caL}{\mathcal L}
\newcommand{\caB}{\mathcal B}
\newcommand{\caO}{\mathcal O}

\newcommand{\G}{G}
\newcommand{\GR}{G(\R)}
\newcommand{\HR}{H(\R)}
\newcommand{\GC}{G(\C)}
\newcommand{\HC}{H(\C)}
\renewcommand{\H}{\mathbb H}
\newcommand{\h}{\mathfrak h}

\newcommand{\tW}{\widetilde W}
\newcommand{\tw}{\tilde w}

\newcommand{\ch}[1]{#1^\vee}
\newcommand{\chG}{\ch{G}}

\renewcommand{\sec}[1]{\section{#1}
\renewcommand{\theequation}{\thesection.\arabic{equation}}
  \setcounter{equation}{0}}
\newcommand{\subsec}[1]{\subsection{#1}
\renewcommand{\theequation}{\thesection.\arabic{equation}}}

\newcommand{\g}{\mathfrak g}
\newcommand{\GGamma}{G^\Gamma}

\newcommand{\chGGamma}{G^{\vee\Gamma}}

\newcommand{\weil}{W_\R}
\newcommand\inv{^{-1}}
\newcommand\bs{\backslash}
\newcommand\wt{\widetilde}

\begin{document}
\title{Algorithms for Representation Theory of Real Reductive Groups
}
\author{Jeffrey Adams and Fokko du Cloux}
\maketitle
{\renewcommand{\thefootnote}{}
\footnote{This work was supported in part  National Science
Foundation Grants \#DMS-0554278 and \#DMS-0532393.}}

\noindent{\bf\Large Introduction}

The irreducible admissible representations of a real reductive group
such as $GL(n,\R)$ have been
classified by work of Langlands, Knapp, Zuckerman and Vogan.  This
classification is somewhat involved and requires a substantial number
of prerequisites. See \cite{overview} for a reasonably accessible
treatment. It is fair to say that it is difficult for a non-expert to
understand any non-trivial case, not to mention a group such as $E_8$.

The purpose of these notes is to describe an algorithm to compute the
irreducible admissible representations of a real reductive group.
The algorithm has been implemented on a computer by the second
author.  This work is part of the Atlas of Lie Groups and
Representations project.  An early version of the software (Version
0.3 as of July 2008), and other documentation and information, may be
found on the web page of the Atlas project, {\tt www.liegroups.org}.

Here is some more detail on what the algorithm and the software do:

\begin{enumerate}
\item Allow the user to define
\begin{enumerate}
\item A complex reductive group $G$,
\item An inner class of real forms of $G$,
\item A particular real form $\GR$ of $G$.
\end{enumerate}
\item Enumerate the Cartan subgroups of $\GR$, and describe them as
  real tori,
\item For any Cartan subgroup $H(\R)$  compute $W(\GR,H(\R))$ (the ``real'' Weyl group),
\item Compute a set $\X$ parametrizing the set $K\bs G/B$ where $B$ is a
  Borel subgroup and $K$ is the complexification of a  maximal compact subgroup of $\GR$,
\item Compute a set $\caZ$ parametrizing the irreducible
  representations of $\GR$ with regular integral infinitesimal
  character,
\item Compute the cross action and Cayley transforms on $\X$ and $\caZ$,
\item Compute Kazhdan-Lusztig polynomials.
\end{enumerate}

In fact the proper setting for all of the preceding computations is
not a single real group $\GR$, but an entire inner class of real
forms, as described in Sections \ref{s:groups}--\ref{s:basic}.

The approach used in these notes most closely follows
\cite{bowdoin}.
 This reference has the advantage over \cite{abv},
which later supplanted it, in that it focuses on the case of regular
integral infinitesimal character, and avoids some 
complications arising from the general case. There are a few changes
in terminology from these references which are discussed in the
remarks. We also make extensive use of  \cite{av1}.
This work has some overlap with, and was influenced by, that of 
Richardson and Springer \cite{richardson_springer}.

A number or examples illustrating the algorithm and the software may
be found in 
{\it Guide to the  Atlas Software: Computational Representation  Theory of Real Reductive Groups}
\cite{adams_snowbird}. The reader is encouraged 
to use the software, and consult \cite{adams_snowbird},  while reading
this paper. An earlier version of this paper by the
second author is available on the Atlas web site \cite{combinatorics}.

The authors thank David Vogan for his numerous contributions to this
project, and Marc van Leeuwen, Patrick Polo and the referee 
for very helpful suggestions.

\medskip
Fokko du Cloux died  of ALS in November 2007. Fokko played a
major role in development of the algorithm described in this
paper, in addition to writing the atlas software. This work began in
2004 and was mostly complete by late 2006. Fokko was an active
participant in the writing of this paper, which was originally
submitted shortly after his death.

\bigskip

\noindent{\bf Index of Notation}

\medskip\noindent
$\Aut(G),\Int(G),\Out(G)$: the (holomorphic) automorphisms of $G$, the
inner automorphisms, and the outer automorphisms
$\Aut(G)/\Int(G)$, respectively (Section \ref{s:automorphisms}). 

\medskip\noindent
$\caB$: set of Borel subgroups of $G$ (Section \ref{s:kgb}), in
bijection with $G/B$.

\medskip\noindent 
$(B,H,\{X_\alpha\})$: splitting  or pinning data for $G$ (Section
\ref{s:automorphisms}). Defines a splitting $\Out(G)\rightarrow
\Aut(G)$ (cf. \eqref{e:exact}).

\medskip\noindent
$c^\alpha,c_\alpha$: noncompact imaginary and real Cayley transforms,
respectively (Definitions \ref{d:cayley} and \ref{d:cayley2}).

\medskip\noindent 
$D=(X,\Delta,\ch X,\ch \Delta)$: root datum defining a complex
connected reductive group \eqref{e:rootdatum}.

\medskip\noindent 
$D(G,H)$: root datum associated to a complex connected reductive group
$G$ and a Cartan subgroup $H$ \eqref{e:DGH}.

\medskip\noindent 
$D_b=(X,\Pi,\ch X,\ch\Pi), D_b(G,B,H)$: abstract {\it based} root datum, and 
the based root datum associated to $G$, $H$ and a Borel subgroup $B$
\eqref{e:DbGBH}.

\medskip\noindent 
$\ch D,\ch D_b,\ch G$: dual (based) root datum and dual group
associated to $D,D_b$ and $G$ respectively (end of Section \ref{s:groups}).

\medskip\noindent
$(G,\gamma)$: basic data consisting of a complex connected reductive
group $G$ and an involution $\gamma\in\Out(G)$ (Section \ref{s:basic}).

\medskip\noindent
$\GGamma$: given basic data $(G,\gamma)$, the 
extended group containing $G$; $\GGamma=G\rtimes\Gamma$ 
where $\Gamma=\{1,\sigma\}$ and $\sigma$ acts by a distinguished
involution mapping to $\gamma$
(Definition \ref{d:extended}).

\medskip\noindent
$(\ch G,\ch \gamma)$:  basic data dual to $(G,\gamma)$ (Section \ref{s:basic}).

\medskip\noindent
$H'_{-\tau}, H_{-\tau}, A_\tau, T_\tau$: various subgroups of the Cartan
subgroup $H$,
depending on a twisted involution $\tau$ (\eqref{e:H} and Remark
\ref{r:H2}). In particular $A_\tau\subset H_\tau\subset H'_\tau$,
$A_\tau$ and $T_\tau$ are connected complex tori, and
$H=T_\tau A_\tau$. 

\medskip\noindent
$\I(G,\gamma)$: set of strong involutions in $\GGamma$ 
(Definition \ref{d:stronginvolution}).

\medskip\noindent
$I$: $\{\xi_i\,|\, i\in I\}$ is a set of representatives of
the set of strong real forms, i.e.
$\I(G,\gamma)/G$ \eqref{e:I1}.

\medskip\noindent
$\I_i$: for $i\in I$,  strong involutions conjugate to $\xi_i$
\eqref{e:I}.

\medskip\noindent
$\I_W$: twisted involutions in $W$  $=\{\tau\in W^\Gamma\bs
W\,|\,\tau^2=1\}$ \eqref{e:I_W} and \eqref{e:concrete}.

\medskip\noindent
$\caL,\caL_c$: equivalance classes of (complete) two-sided L-data for
$(G,\gamma)$ \eqref{e:L}. 

\medskip\noindent
$N,N^\Gamma:$ normalizer of $H$ in $G$ and $\GGamma$, respectively 
(\eqref{e:x1}).

\medskip\noindent
$\tilde p,p$: maps from $\tX$ and $\X$ to $\I_W$, respectively
(\eqref{e:p}).

\medskip\noindent
$P,\ch P,P_{reg},\ch P_{reg}, P(G,H), \ch P(G,H)$: weight and coweight
lattices and their regular elements (Section \ref{s:basic}).

\medskip\noindent
$\caP,\caP_c$: equivalence classes of  (complete) one-sided L-data for $(G,\gamma)$)
\eqref{e:caPchP};
$\caP=(\I\times\caB)/G$ (\eqref{e:caPIBG}).
Also $\ch\caP,\ch\caP_c$ on the dual side. 

\medskip\noindent
$\caP[\xi]$: set of equivalence classes of one-sided L-data  $(\xi',B)$
such that $\xi$ is conjugate to $\xi$
(\eqref{e:caPxi}). In bijection with $K_\xi\bs G/B$
(\eqref{e:capxi}).

\medskip\noindent
$S=(\xi,B)$: one-sided L-datum for $(G,\gamma)$, consisting of 
of a strong involution $\xi$ and 
a Borel subgroup $B$ of $G$ (Definition \ref{d:onesidedldatum}).
Also $\ch S=(\eta,\ch B_1)$ is a one-sided L-datum for $(\ch G,\ch\gamma)$.

\medskip\noindent
$\ch S_c=(\eta,\ch B_1,\lambda)$: complete one-sided L-datum for $(\chG,\ch\gamma)$, consisting of
an L-datum $(\eta,\ch B_1)$ 
and $\lambda\in\ch\h$ satisfying
$\exp(2\pi i\lambda)=\eta^2$
(Definition \ref{d:onesidedcompleteldatum}).

\medskip\noindent
${\bf S}=(S,\ch S)$: two-sided L-datum, where $S,\ch S$
are L-data for $(G,\gamma)$ and $(\ch G,\ch\gamma)$ respectively,
satisfying an additional compatibility condition
(Definition \ref{d:L}).

\medskip\noindent
${\bf S_c}=(S,\ch S_c)$: complete two-sided L-datum, where 
$S$ is an L-datum for $(G,\gamma)$, $\ch S_c$ is a complete L-datum for $(\ch
G,\ch\gamma)$,
satisfying an additional compatibility condition
(Definition \ref{d:L}).

\medskip\noindent
$W,W^\Gamma$: $W=N/H$ is the Weyl group, and $W=N^\Gamma/H$ is the
extended Weyl group 
\eqref{e:WGamma}.

\medskip\noindent
$\tW$: the Tits group; a subgroup of $G$
mapping onto $W$
(Definition \ref{d:tits}).

\medskip\noindent
$\X=\X(G,\gamma)$: the one-sided parameter space 
$=\tX/H=\{\xi\in\Norm_{\GGamma\bs G}(H)\,|\, \xi^2\in
Z(G)\}/H$
(Definition \ref{d:X}).

\medskip\noindent
$\X^r=\X^r(G,\gamma)$: the {\it reduced} parameter space; a finite
subset of $\X$ (Definition \ref{d:reduced}).

\medskip\noindent
$\tX$: strong involutions normalizing $H$
$=\{\xi\in\Norm_{\GGamma\bs G}(H)\,|\, \xi^2\in Z(G)\}$
\eqref{e:Xt}.

\medskip\noindent
$\X_\tau,\tX_\tau$: for $\tau$ a twisted involution, the fiber of the map
$p:\X\rightarrow \I_W$ 
(resp. $\tilde p:\tX\rightarrow \I_W$)
over $\tau$
(beginning of Section \ref{s:fibers}).

\medskip\noindent
$\X[x]$: for $x\in\X$, subset of $\X$ consisting of elements conjugate to $x$
\eqref{e:Xx}.

\medskip\noindent
$\X(z)$: for $z\in Z(G)$ the  elements of $\X$ satisfying $x^2=z$
(\eqref{e:Xz}).

\medskip\noindent
$\X_\tau(z)$: for  $z\in Z(G),\tau\in\I_W$,  $\X_\tau(z)=\X_\tau\cap \X(z)$; 
the  elements $x\in\X$ satisfying $p(x)=\tau$ and $x^2=z$
(beginning of Section \ref{s:fibers}).

\medskip\noindent
$\caZ=\caZ(G,\gamma)$: the two-sided parameter space, contained in
$\X(G,\gamma)\times\X(\ch G,\ch\gamma)$ (Definition \ref{d:Z}).

\medskip\noindent
$\delta$: distinguished element of $\X$ or $\I_W$ (Definition
\ref{d:extended} and before Lemma \ref{l:pX}).

\medskip\noindent
$\Delta_i,\Delta_r,\Delta_{cx},W_i,W_{i,\tau}, W_r,W_{r,\tau}$, etc.: 
various root systems and Weyl groups (beginning of Section
\ref{s:waction}). 

\medskip\noindent 
$\ch\gamma$: element of   $\Out(\ch G)$ corresponding to
$\gamma\in\Out(G)$ (Defintion \ref{d:chtau}).

\medskip\noindent
$\theta_\xi,K_\xi$: for $\xi\in\tX$, $\theta_\xi=\int(\xi)$, an involution of $G$, and
$K_\xi=G^{\theta_\xi}$
(Definition \ref{d:stronginvolution}).
Thus $\theta_\xi$ is the Cartan involution of a real form of $G$, and
$K_\xi$ is the corresponding complexified maximal compact subgroup
(Section \ref{s:involutions}). 

\medskip\noindent
$\theta_i,K_i$: For $i\in I$, Cartan involution and complexified maximal
compact subgroup associated to $\xi_i$ \eqref{e:I}.

\medskip\noindent
$\xi:$ strong involution (with respect to $(G,\gamma)$); 
$\xi\in\GGamma\bs G$ satisfies $\xi^2\in Z(G)$ (Definition
\ref{d:stronginvolution}).
Then $\theta_\xi=\int(\xi)$ is an involution of $G$.

\medskip\noindent
$(\xi,\pi)$: Harish-Chandra module for a strong involution; 
a strong involution $\xi$ and a $(\g,K_\xi)$-module $\pi$.  
A {\it Harish-Chandra module for a strong real form} is a
$G$-orbit of such pairs (Section \ref{s:representations}).

\medskip\noindent
$\Pi(G(\R),\lambda)$: set of equivalence classes of irreducible
admissible representations of $G(\R)$ with infinitesimal character
$\lambda$
(Definition \ref{d:infchar}).

\medskip\noindent
$\Pi(G(\R),\Lambda)$: for $\Lambda\subset P_{reg}$,
$\coprod_{\lambda\in\Lambda}\Pi(G(\R),\lambda)$
(before Theorem \ref{t:ldata}).

\medskip\noindent
$\Pi(G,\xi)$: equivalence classes of $(\g,K_\xi)$-modules with
regular integral infinitesimal character (Definition \ref{d:pig}).

\medskip\noindent
$\Pi(G,\gamma)$: representations of strong real forms of $G$ with
regular integral infinitesimal character; 
$=\{(\xi,\pi)\,|\,\xi \in \I,\pi\in\Pi(G,\xi)\}/G$
(Definition \ref{d:pig}).

\medskip\noindent
$\Pi(G,\gamma,\Lambda)$: for $\Lambda\subset P_{reg}$, the  subset of $\Pi(G,\gamma)$ consisting of
representations with infinitesimal character  contained in
$\Lambda$. 

\medskip\noindent
$\Pi_\phi(G,\xi)$: L-packet of representations of strong real form
$\xi$ associated to the L-homomorphism $\phi$
(after \ref{e:PiG}).

\medskip\noindent
$\Pi_\phi(G,\gamma)$: {\it large} L-packet; union of L-packets for all strong real forms
\eqref{e:large}. 

\medskip\noindent
$\Pi_{\ch S_c}(G,\gamma)$: large L-packet $=\Pi_\phi(G,\gamma)$ where 
$\phi$ is defined by \eqref{e:phizj} (cf.~\eqref{e:large2}).

\sec{Overview}
\label{s:overview}

The primary aim of this paper is to distill a well-known but difficult
theory into an algorithm which can be implemented on a computer.
While the resulting algorithm is self-contained and comparatively
elementary, even understanding the algorithm itself requires a
fairly deep knowledge of the mathematics.  In this section we give a
high level overview of the algorithm, before going into more detail in
the remainder of the paper.
An outline of the contents of the paper
appears at the end of this Overview (Section \ref{s:outline}).

We assume the reader is familiar with the theory of admissible
representations of real reductive groups. A good introduction  is
Knapp's book \cite{overview}.

In this section we write  $\GC$ for a complex reductive group, with real
points $\GR$.
We will make various simplifying assumptions in the course of this
overview; this is for ease of exposition, and the general statements
will be found in the body of the paper. See the end of this section
for a discussion of the issues involved.

Let $\g$ be the Lie algebra of $\GC$.
Fix a Cartan involution $\theta$ of $\GC$ corresponding to $\GR$.
Thus $K=\GR^\theta$ is a maximal compact subgroup of $\GR$. The
basic goal is to parametrize the irreducible admissible
representations of $\GR$, or equivalently  the irreducible admissible
$(\g,K)$-modules.  This is an infinite (typically uncountable) set.

Suppose $\HC$ is a Cartan subgroup of $\GC$, and let $\h=\Lie(\HC)$.
By the Harish-Chandra homomorphism, associated to $\lambda\in\h^*$ is
an {\it infinitesimal character} which we also denote $\lambda$. 
We say $\lambda$ is {\it regular} (resp. {\it integral}) if 
$\langle\lambda,\ch\alpha\rangle\ne 0$  (resp. $\in\Z$) 
for all roots $\alpha$ of $\h$ in $\g$.

Since we will be working with different Cartan subgroups
simultaneously, it is convenient to fix one, denoted $H_a(\C)$, the
{\it abstract} Cartan subgroup. Then we fix once and for all 
$\lambda\in\h_a^*$.
The first basic reduction is to consider 
representations with infinitesimal character $\lambda$. 

\begin{definition}
\label{d:infchar}
Fix $\lambda\in \h_a^*$, and let
$\Pi(\GR,\lambda)$ be the set of irreducible admissible
representations of $\GR$ with infinitesimal character $\lambda$.  
\end{definition}

By a result of Harish-Chandra this is a finite set.  The main result
of this paper is an algorithm to compute $\Pi(\GR,\lambda)$ when
$\lambda$ is regular and integral.

There are a number of approaches to classifying $\Pi(\GR,\lambda)$. 
We focus on three of them, which are 
intertwined with each other, and each of which plays a role in the algorithm.

\medskip

\subsection{The Langlands classification}
\label{s:langlands}
Many representations of reductive groups can be constructed using
characters of Cartan subgroups.
For example if $G$
is a reductive group over a finite field  these are the
representations $R_T(\theta)$ of Deligne and Lusztig \cite{deligne_lusztig}.
For generic $\theta$ $R_T(\theta)$ is irreducible.

For another example, suppose $\F$ is a finite or local field, $G=G(\F)$ is split, and
$B=HN$ is a Borel subgroup of $G$. If $\chi$ is a character of $H$
then $\Ind_B^G(\chi\otimes 1)$ is a minimal principal series
representation of $G$; again for generic $\chi$ this is
irreducible.

For a final example assume $\GC$ is semisimple and simply connected, and $G=\GR$
is a real form of $\GC$, containing a compact Cartan subgroup $T$. 
Suppose $\chi$ is a character of $T$ such that $\langle
d\chi,\ch\alpha\rangle\ne0$ for all roots $\alpha$.
Associated to 
$\chi$ is the discrete series representation of $G$  with Harish-Chandra
parameter $d\chi$, and every discrete series representation of $G$
is of this form.

The Langlands classification for $\GR$ is built out of the second two cases.
For now we  assume  that $\GC$ is acceptable, 
i.e. $\rho$ (one half the sum of a set of positive roots) exponentiates to
$H_a(\C)$.  This holds for example if $\GC$ is semisimple and simply
connected. 

Consider the set of pairs $(H(\R),\chi)$ where $H(\R)$ is a
Cartan subgroup of $\GR$, $\chi$ is a character of $H(\R)$, and  $d\chi$ is
$\GC$-conjugate to $\lambda$. The group $\GR$ acts on these pairs by 
conjugation. We define {\it character data} as follows:
\begin{equation}
\caC(\GR,\lambda)=
\{(\HR,\chi)\,|\, d\chi\text{ is $\GC$-conjugate to }\lambda\}/\GR.
\end{equation}

\begin{proposition}[Langlands Classification]
\label{p:LC}
Assume $G(\C)$ is acceptable and $\lambda$ is regular. 
There is a natural bijection 
\begin{equation}
\label{e:Hchi}
\Pi(\GR,\lambda)\overset{1-1}\longleftrightarrow \caC(\GR,\lambda).
\end{equation}
\end{proposition}
There are many versions of the Langlands classification, for example see
\cite{overview} or \cite{green}. This
restatement of (a special case of) the 
classification is taken from \cite[Theorem 8.29]{av1}. 
Since $\GC$ is acceptable we do not need the
$\rho$-cover of $H$, and since $\lambda$ is regular  we
can take $\Psi$ to be the set of positive,  real integral roots defined by
$\lambda$ (notation as in \cite{av1}).

Thus one version of our algorithm would be to replace the right hand
side of \eqref{e:Hchi} by a computable combinatorial
object.
Here is an idea of what this involves.

First we fix a set $H_1(\R),\dots, H_n(\R)$ of representatives of the
conjugacy classes of Cartan subgroups. 
An algorithm for computing this set is given by 
Kostant\cite{kostant_cartan}, \cite{kostant_cartan_2}, and 
independently by  Borel (unpublished) and Sugiura \cite{sugiura_cartans}.

Now fix $i$ and let $H(\R)=H_i(\R)$.
Let $W(\GR,H(\R))$ be the ``real'' Weyl group
$\Norm_{\GR}(H(\R))/H(\R)$. This is a subgroup of the  Weyl group
$W(\GC,H(\C))=W(\g,\h)$.
Unlike $W(\g,\h)$, $W(\GR,\HR)$ depends on $i$,
is not
necessarily the Weyl group of a root system,
and can be somewhat difficult to compute. An algorithm is given in 
\cite{knapp_weyl_group}; also see \cite[Proposition 4.16]{ic4}.

\begin{lemma}[\cite{green}, Theorem 2.2.4]
\label{l:number}
In the setting of Proposition \ref{p:LC} there is a bijection 

\begin{equation}
\label{e:WWRchi}
\Pi(\GR,\lambda)\overset{1-1}\longleftrightarrow
\coprod_i^n (W/W(\GR,H_i(\R))\times [H_i(\R)/H_i(\R)^0]\sphat.
\end{equation}
In particular
\begin{equation}
\label{e:number}
|\Pi(\GR,\lambda)|=\sum_{i=1}^n|W/W(\GR,H_i(\R))||H_i(\R)/H_i(\R)^0|.
\end{equation}
\end{lemma}

See \cite[Theorem 2.2.4]{green}.

The bijection \eqref{e:WWRchi} depends on a number of choices. In addition to the
choice of the $H_i(\R)$, for each $i$ we have to choose a Borel
subgroup containing $H_i(\C)$, and an embedding of $H_i(\R)/H_i(\R)^0$
in $H_i(\R)$.
Nevertheless this result gives a good idea of what we need to compute:

\begin{enumerate}
\label{en:need}
\item the Cartan subgroups $H_1(\R),\dots, H_n(\R)$; and for each $i$,
\item $W(\GR,H_i(\R))$,
\item $H_i(\R)/H_i(\R)^0$.
\end{enumerate}
In any event it is clear that any parametrization of
$\Pi(\GR,\lambda)$ must (if only implicitly) include a
description of items (1-3).

\subsection{$\caD$-modules}
\label{s:dmodules}
We now describe  the classification of $\Pi(\GR,\lambda)$ in terms of
$\caD$-modules. The basic reference is 
\cite{bbd}, or see \cite{milicic_dmodules} 
for a  good introduction to the subject.

Recall $\theta$ is a Cartan involution of $\GC$
corresponding to $\GR$. Let $K(\C)=\GC^\theta$. Then $K(\C)$ is a
reductive group (possibly disconnected), with real points
$K=\GR^\theta$.
By the relationship between finite dimensional representations of $K$
and $K(\C)$, we may view  admissible $(\g,K)$-modules as 
$(\g,K(\C))$-modules. 

Let $\caB$  be the variety of Borel subgroups of $\GC$,
so $\caB=\GC/B(\C)$ where $B(\C)$ is a fixed Borel subgroup. Then
$K(\C)$ acts on $\caB$ with finitely many orbits.
Let $\caD_\lambda$ be the sheaf of twisted differential operators on
$\caB$ corresponding to $\lambda$. 
We consider  {\it $\caD$-module data}:

\begin{equation}
\caD(\GR,\lambda)=\{(\caO,\tau)\}/K(\C).
\end{equation}
where $\caO$ is a $K(\C)$-orbit on $\caB$, and $\tau$ is an irreducible 
$K(\C)$-equivariant $\caD_\lambda$-module.

Here is how to make $\tau$ more concrete. Fix $x\in\caO$ and let
$K_x(\C)=\Stab_{K(\C)}(x)$. Let $B(\C)$ be the Borel subgroup of
$\GC$ corresponding to $x$, and let $H(\C)$ be a
$\theta$-stable Cartan subgroup contained in $B(\C)$.
Then $\tau$ can
be viewed as a character of $H(\R)$ 
whose differential $d\chi$ is
$\GC$-conjugate to $\lambda$.

\begin{proposition}[\cite{bbd}; \cite{milicic_dmodules}, Theorem 3.9]
\label{t:caD}
  There is a natural bijection
\begin{equation}
\Pi(\GR,\lambda)\overset{1-1}\longleftrightarrow\caD(\GR,\lambda).
\end{equation}
\end{proposition}
It is reasonable to look for an algorithm to compute the right hand
side.

Fix an orbit $\caO$ of $K(\C)$ on $\caB$.
By the Proposition, for every
$\caD_\lambda$-module $\tau$ on $\caO$ we obtain an irreducible
representation. Varying $\tau$ we obtain a finite set of representations
associated to $\caO$.
For (not very good) reasons which will become clear in the next
section (also see \cite[Section 8]{ic4}) we refer to 
this set as an 
{\it $R$-packet}, and denote it 
$\Pi_R(\GR,\caO,\lambda)$.
Thus $\Pi(\GR,\lambda)$ is a disjoint union of R-packets:
\begin{equation}
\label{e:rpackets}
\Pi(\GR,\lambda)=\coprod_{K(\C)\bs \caB} \Pi_R(\GR,\caO,\lambda).
\end{equation}

\subsection{The Langlands classification using L-groups}
\label{s:langlands2}
We now consider a version of the Langlands classification in terms
of L-groups.  Given our group $\GC$, with real points $\GR$, let
$\chGGamma$ be the L-group of $\GR$ \cite{borel_corvallis}.  Thus
$\chGGamma=\ch G(\C)\rtimes \Gamma$, where $\ch G(\C)$ is the dual group
of $\GC$, $\Gamma=\text{Gal}(\C/\R)$, and $\chGGamma$ is a certain
semidirect product. For example if $\GR$ is split this is a
direct product.

Let $W_\R$ be the Weil group of $\R$,
i.e. $W_\R=\langle\C^\times,j\rangle$ where $jzj\inv=\overline z$ and
$j^2=-1$. We consider the space 
$\Homadm(W_\R,\chGGamma)$
of {\it admissible homomorphisms} of
$W_\R$ into $\chGGamma$. These are the continuous homomorphisms such that
$\phi(\C^\times)$ consists of semisimple elements of $\ch G(\C)$, and
$\phi(j)\not\in \ch G(\C)$.

The version of the Langlands classification in terms of L-groups
says that associated to an
admissible homomorphism $\phi$ is a finite set of irreducible
admissible representations $\GR$, called an 
{\it L-packet}.
The L-packet of $\phi$ only depends on the $\ch
G(\C)$-orbit of $\phi$, and we write 
$\Pi_L(\GR,\ch\caO)$ for the L-packet associated to such an orbit
$\ch\caO$.
An L-packet may be empty (if $G(\R)$ is not quasi-split); 
non-empty L-packets associated to two orbits
are equal if only if the orbits are equal.
Finally the admissible dual of $\GR$ is the disjoint union of L-packets.

The representations in $\Pi_L(\GR,\ch\caO)$ all have the same infinitesimal
character, which is determined by $\ch\caO$ (see Section \ref{s:ldata}).
Let  $\Homadm(W_\R,\chGGamma,\lambda)$ be 
the admissible homomorphisms  for which $\Pi_L(\GR,\phi)$ has
infinitesimal character $\lambda$.
Thus (compare \eqref{e:rpackets}) $\Pi(\GR,\lambda)$ is a disjoint union of
L-packets:  
\begin{equation}
\label{e:alllpackets}
\Pi(\GR,\lambda)=\coprod_{\ch\caO}\Pi_L(\GR,\ch\caO)
\end{equation}
where  the union runs over $\ch G(\C)$ orbits on 
$\Homadm(W_\R,\chGGamma,\lambda)$.

Suppose an L-packet $\Pi_L(\GR,\ch\caO)$ is non-empty. Some additional data is required to specify a
particular representation in $\Pi_L(\GR,\ch\caO)$.
Choose $\phi\in \ch\caO$, let $S_\phi$ be the centralizer in $\ch G(\C)$ of the image of $\phi$,
and let $\mathbb S_\phi=S_\phi/S^0_\phi$. Roughly speaking
$\Pi_L(\GR,\ch\caO)$ should be parametrized by characters of $\mathbb S_\phi$.
This leads us to define the  set of {\it Langlands data}
\begin{equation}
\caL(\chGGamma,\lambda)
=\{(\phi,\chi)\,|\, \phi\in
\Homadm(W_\R,\chGGamma,\lambda), 
\chi\in \widehat{\mathbb S_\phi}\}/\ch G(\C).
\end{equation}

Suppose $\G_1(\R), \G_2(\R)$ are two real forms of $\GC$. 
It may be that the associated L-groups are isomorphic (we say the real
forms are inner to each other if this holds). 
For example this is the case for all real forms of $\GC$ if $\GC$ is
semisimple and its Dynkin diagram admits no non-trivial
automorphisms.
Unlike  $\caC(\GR,\lambda)$ and $\caD(\GR,\lambda)$,
$\caL(\chGGamma,\lambda)$ should be related to 
representations of all these real forms. 

To make this precise it is convenient to assume that $\GC$ is
adjoint. Let $\G_1(\R),\dots, \G_n(\R)$ be the inequivalent real forms
of $\GC$ with L-group $\chGGamma$. 

\begin{proposition}[\cite{bowdoin}, Theorem 3-2]
\label{t:caL}
Assume $G(\C)$ is adjoint. There is a natural bijection
\begin{equation}
\coprod_{i=1}^n \Pi(\G_i(\R),\lambda)
\overset{1-1}\longleftrightarrow\caL(\chGGamma,\lambda).
\end{equation}
\end{proposition}
Again it is reasonable to look for an algorithm to compute the right
hand side.

\subsection{L-packets and R-packets}
\label{s:packets}

Fix a real form $\GR$ of $\GC$.
Consider for the moment the problem of explicitly parametrizing
$\Pi(\GR,\lambda)$ using L-homomorphisms. 

Recall \eqref{e:alllpackets}
$\Pi(\GR,\lambda)$ is  the disjoint union of L-packets
$\Pi_L(\GR,\ch\caO)$ (some of these may be empty).
Specifying a single representation in an L-packet $\Pi_L(\GR,\ch\caO)$ amounts to
specifying a character of the  two-group $\mathbb S_\phi$ ($\phi\in\ch\caO$), a
problem we prefer to avoid.

On the other hand $\Pi(\GR,\lambda)$ is also the disjoint union of
R-packets $\Pi_R(\GR,\caO,\lambda)$ (see \eqref{e:rpackets}).
Again specifying a single representation in an R-packet requires
specifying a character of a two-group, in this case the component
group of a torus.

The key to our parametrization is that the intersection of an L-packet
and an R-packet is at most one element.
We thereby obtain a classification in terms of pairs of packets, i.e. pairs of
orbits. 
Here is a weak version of this result, which does not require any
further assumptions on $\GC$:

\begin{lemma}[\cite{ic4}, Proposition 8.3]
\label{l:Ophi}
Suppose  $\Pi_R(\GR,\caO,\lambda)$ is an R-packet, and 
$\Pi_L(\GR,\ch\caO)$ is an L-packet.
Then $\Pi_L(\GR,\ch\caO)\cap \Pi_R(\GR,\caO,\lambda)$ has at most one
element.
\end{lemma}

This reduces the problem of parametrizing $\Pi(\GR,\lambda)$ to the
following problems. Let $K(\C)$ be the complexification of a  maximal compact
subgroup of $\GR$.
\begin{enumerate}
\item Parametrize $K(\C)$-orbits on $\caB$,
\item Parametrize $\ch G(\C)$-orbits on $\Homadm(W_\R,\chGGamma,\lambda)$.
\end{enumerate}
It turns out that (2) is equivalent to
\begin{enumerate}
\item[(2$'$)]
Parametrize $\ch K(\C)$-orbits on $\ch\caB$.
\end{enumerate}
Here $\ch\caB$ is the variety of Borel subgroups of $\chG(\C)$ and $\ch
K(\C)$ is the fixed points of an involution of $\chG(\C)$, i.e. the
analogue of (1) on the dual side.
We also need to determine when the intersection of an L-packet and an
R-packet is non-empty.

We therefore turn  to the problem of computing the space of $K(\C)$
orbits on $\caB$, before returning the parametrization of
$\Pi(\GR,\lambda)$ in Section \ref{s:overview:Z}.

\subsection{Parametrization of $K$ Orbits on the Flag Variety}
\label{s:X}
Fix $\GC$. As in Section \ref{s:dmodules} we are interested in computing the space of  $K(\C)$-orbits
on $\caB$, where $K(\C)$ is the fixed points of a Cartan involution of
$\GC$. It turns out it is easier to treat all Cartan involutions
simultaneously. 
For the purposes of this Overview we make two simplifying assumptions:
\begin{enumerate}
\item the center of $\GC$ is trivial,
\item every automorphism of $\GC$ is inner.
\end{enumerate}
Assuming (1), condition (2) is equivalent to:

\begin{enumerate}
\item[(2$'$)]the Dynkin diagram of $\GC$ has no nontrivial
automorphisms.
\end{enumerate}
For example this holds if $\GC$ is a simple adjoint group of type 
$B_n$,$C_n$,$E_7$, $E_8$,$F_4$ or $G_2$.

Under assumption (2) every involutive automorphism of $\GC$ is of the form $\int(x)$ for some
involution $x\in\GC$ (where $\int(x)$ is conjugation by $x$).
It follows that the (equivalence classes of) real forms of
$G(\C)$ are in bijection with conjugacy classes of 
involutions in $\GC$. If $x$ is such an involution then
$\theta_x=\int(x)$ is an involution of $G(\C)$, and is the Cartan
involution of a real form. Conversely the Cartan involution of every
real form is of the form $\theta_x$. See Section \ref{s:involutions}.

\begin{definition}[Definition \ref{d:X}]
Assuming (1) and (2), define
\label{d:overview:X}
\begin{equation}
\label{e:overview:X}
\X(\GC)=\{x\in \Norm_{\GC}(\HC)\,|\, x^2=1\}/\HC.
\end{equation}  
\end{definition}
The quotient is by conjugation by $\HC$. This is a finite set.

Let $\{x_1,\dots, x_n\}$ be a set of representatives of the
conjugacy classes of involutions of $\GC$.
For $1\le i\le n$ let $\theta_i=\int(x_i), K_i(\C)=\GC^{\theta_i}$ and 
write $\G_i(\R)$ for the corresponding real form of $\GC$.
Fix a Cartan subgroup $H(\C)$.

\begin{proposition}
\label{p:overview:XGC}
Assume  (1) and (2). 
There is a natural bijection
\begin{equation}
\label{e:XGC}
\X(\GC)\overset{1-1}\longleftrightarrow \coprod_{i=1}^n K_i(\C)\bs \caB.
\end{equation}
\end{proposition}

\begin{proof}[Sketch of proof]
Let 
\begin{equation}
\caP=\{(x,B(\C))\,|\, x\in \GC, x^2=1, B(\C)\text{ a Borel subgroup}\}/\GC.
\end{equation}
Every element of $\caP$ is conjugate to one of the form
$(x_i,B(\C))$; the set of conjugacy classes of pairs $(x_i,B(\C))$ is isomorphic to
$K_i(\C)\bs \caB$. This gives a bijection
$\caP\overset{1-1}\longleftrightarrow\coprod_{i=1}^n K_i(\C)\bs \caB$.

On the other hand
fix a Borel subgroup $B_0(\C)$
containing $H(\C)$. 
Every element of $\caP$ is
conjugate to one of the form $(x,B_0(\C))$.
Furthermore by conjugating by $B_0(\C)$ we may assume
$x\in \Norm_{\GC}(H(\C))$, which gives a bijection
$\caP\overset{1-1}\longleftrightarrow \X(\GC))$.
See Section \ref{s:kgb}.
\end{proof}

We turn next to the computation of
$\Homadm(W_\R,\chGGamma,\lambda)/\ch G(\C)$. A remarkable fact, mentioned at
the end of the previous section, is that the space $\X$ (applied on
the dual side) provides this parametrization. 
To make this precise it is convenient to
assume that the center of $\chG(\C)$ is trivial. 

Fix a Cartan subgroup $\ch H(\C)$ of $\chG(\C)$. Applying Definition
\ref{d:overview:X} to $\chGGamma$ we have

\begin{equation}
\X(\chG(\C))
=\{x\in \Norm_{\chG(\C)}(\ch H(\C))\,|\, x^2=1\}/\ch H(\C).
\end{equation}

\begin{proposition}
\label{p:overview:XchGC}
Assume $\ch G(\C)$ is adjoint.
There is a natural bijection
\begin{equation}
\X(\chG(\C))\overset{1-1}\longleftrightarrow\Homadm(W_\R,\chGGamma,\lambda)/\chG(\C).
\end{equation}
\end{proposition}

\begin{proof}[Sketch of proof]
We  may view $\lambda$ as an element of the Lie algebra of $\ch H(\C)$.
If $x\in \Norm_{\ch G(\C)}(\ch H(\C))$, $x^2=1$  define
\begin{equation}
\begin{aligned}
\phi(z)&=z^\lambda\overline z^{\Ad(x)\lambda}\\
\phi(j)&=e^{-\pi i\lambda}x.
\end{aligned}
\end{equation}
Then $\phi\in \Homadm(W_\R,\chGGamma,\lambda)$, and every element 
of $\Homadm(W_\R,\chGGamma,\lambda)$ is conjugate to one of this form.
See Section \ref{s:ldata}.
\end{proof}

Now assume $Z(G(\C))$ and $Z(\chG(\C))$ are trivial, and $\Out(G(\C))=1$. 
Then  Proposition \ref{p:overview:XGC} holds here as well, so
there are bijections
\begin{equation}
\label{e:twobijections}
\Homadm(W_\R,\chGGamma,\lambda)
\overset{1-1}\longleftrightarrow 
\X(\chG(\C))\overset{1-1}\longleftrightarrow \coprod_{i=1}^m \ch K_i(\C)\bs\ch\caB.
\end{equation}
Here $\ch\caB$ is the space of Borel subgroups of $\chG(\C)$, 
$y_1,\dots, y_m$ are representatives of the conjugacy classes of
involutions in $\chG(\C)$, and  for each $i$
$\ch K_i(\C)=\Cent_{\chG(\C)}(y_i)$.

\subsection{The parameter space $\caZ$}
\label{s:overview:Z}

We now combine the results of the previous section with 
Lemma \ref{l:Ophi} to define the parameter space $\caZ$ for
$\Pi(\GR,\lambda)$.

To avoid some technical issues in the previous section
we assumed the center of $\GC$ is trivial and $\GC$ has no outer automorphisms. For
consideration of $\caL(\chGGamma,\lambda)$ we  assumed the
corresponding facts for $\chG(\C)$. 
Therefore we assume $G(\C)$ is both simply connected and adjoint. We
may as well also assume $G(\C)$ is simple, i.e. of type $G_2,F_4$ or
$E_8$. 

As in the previous section let $x_1,\dots, x_n$ be representaitves of
the conjugacy classes of involutions of $\GC$, and let
$K_i(\C)=\Cent_{\GC}(x_i)$.  Dually let $y_1,\dots, y_m$ be
representatives of the conjugacy classes of involutions of $\ch
G(\C)$, and let $\ch K_j(\C)=\Cent_{\ch G(\C)}(y_j)$.

\begin{definition}
\label{d:overview:Z}
Assume $\GC$ is simple, simply connected and adjoint.
Fix Cartan subgroups $\HC$  of $\GC$ and $\ch H(\C)$ of $\ch G(\C)$.
Let $\h,\ch\h$ be the Lie algebras of $\HC$ and $\ch H(\C)$,
respectively. 
Let
\begin{subequations}
\renewcommand{\theequation}{\theparentequation)(\alph{equation}}  
\begin{equation}
\begin{aligned}
\X=\X(\GC)&=\{x\in \Norm_{\GC}(\HC)\,|\,x^2=1\}/\HC\\
\ch\X=\X(\chG(\C))&=\{y\in \Norm_{\chG(\C)}(\ch H(\C))\,|\,y^2=1\}/\ch H(\C).
\end{aligned}
\end{equation}
\end{subequations}
There is a natural adjoint map $\End(\h)\ni X\mapsto X^t\in
\End(\ch\h)$. Let

\begin{equation}
\label{e:adjoint}
\caZ=\{(x,y)\in\X\times \ch\X\,|\, (\Ad(x)|_\h)^t=-\Ad(y)|_{\ch\h}\}.
\end{equation}
\end{definition}

By \eqref{e:XGC} and \eqref{e:twobijections} $\caZ$ 
may be viewed as  a subset of 
\begin{equation}
\label{e:overviewZ}
\coprod_{i=1}^n K_i(\C)\bs\caB
\times
\coprod_{i=1}^m \ch K_i(\C)\bs\ch\caB.
\end{equation}
See \eqref{e:Zorbit}.
Here is a special case of the main result:

\begin{theorem}[Theorem \ref{t:reps}]
\label{t:overviewmain}
Assume $\GC$ is simple, simply connected and adjoint.
Write $G_1(\R),\dots, G_n(\R)$ for the equivalence classes of real
forms of $\GC$.

There is a natural bijection
\begin{equation}
\label{e:caZ}
\caZ
\overset{1-1}\longleftrightarrow
\coprod_{i=1}^n\Pi(G_i(\R),\lambda).
\end{equation}
\end{theorem}

This is a  refinement of Lemma  \ref{l:Ophi}, and 
is a restatement of \cite[Theorem 2-12]{bowdoin}.

\begin{proof}[Sketch of proof]

Here are three ways to think of the proof of this result.

Fix $1\le i\le n$ and $1\le j\le m$. Also fix an orbit 
$\caO$ of $K_i(\C)$ on $\caB$, and an orbit  $\ch\caO$ of $\ch K_j(\C)$ on
$\ch\caB$. 
By Lemma \ref{l:Ophi}
$\Pi_L(G_i(\R),\ch\caO)\cap \Pi_R(G_i(\R),\caO,\lambda)$ consists of at most
element.
Condition \eqref{e:adjoint} makes
this intersection nonempty, and we obtain every representation
exactly once this way.

Alternatively, fix  $1\le i\le m$ and an orbit $\ch\caO$ of $\ch K_i(\C)$ on $\ch\caB$.
Fix $1\le j\le n$,
and let $\Pi_L(\G_j(\R),\ch\caO)$ be the corresponding L-packet.
The choice of an orbit of $K_j(\C)$ on $\caB$
defines a character $\chi$
of $\mathbb S_\phi$, and by Proposition \ref{t:caL} this defines a
representation of $G_j(\R)$.
This is the proof in Section \ref{s:ldata} (see \cite{bowdoin}).

Finally, fix $1\le i\le n$, an orbit $\caO$ of $K_i(\C)$ on $\caB$,
and let $\Pi_R(\G_i(\R),\caO,\lambda)$ be the corresponding R-packet.
The data needed to specify a $\caD_\lambda$ module supported on $\caO$ is precisely an
orbit of $\ch K_j(\C)$ on $\ch\caB$ for some $1\le j\le m$. By Proposition
\ref{t:caD} this defines an irreducible representation of $G_i(\R)$.
\end{proof}

As is clear from the statement, to explicitly parametrize
$\Pi(\GR,\lambda)$ the main issue is to understand the spaces
$\X$ and $\ch\X$.  It is not immediately obvious how to explicitly compute
these, but this can be done using the Tits group. 
This is discussed in Section \ref{s:tits}. 
In any event all of the structural data discussed in
Section \ref{s:overview} can be read off from the space $\X$.

\begin{proposition}[Proposition \ref{p:recapitulation}]
Use the notation of Theorem \ref{t:overviewmain}.

\label{l:or}
\hfil\break
\noindent (1) The real forms of $G$ are parametrized by
$\X\cap H(\C)/W$.
Write $x_1,\dots, x_n$ for representatives of this set.
\medskip

\noindent (2) 
$\X\overset{1-1}\longleftrightarrow\coprod_{i=1}^nK_{x_i}(\C)\bs\caB$.
\medskip

\noindent (3)
Associated to each $x\in\X$ is a pair  $(G_x(\R),H_x(\R))$ 
consisting of a real form of $\GC$  and a Cartan subgroup of $G_x(\R)$.
This induces a bijection between  $\X/W$ and the union, over real
forms $\GR$ of $\GC$, of the conjugacy classes of Cartan
subgroups of $\GR$.

\medskip
\noindent (4) For  $x\in\X$ we have
\begin{equation}
W(G_x(\R),H_x(\R))\simeq \Stab_W(x).
\end{equation}
\end{proposition}

See Section \ref{s:recap}.

It is clear from the discussion that the setting is entirely symmetric
in $\GC$ and $\chG(\C)$.
This is a manifestation of {\it Vogan
duality} \cite{ic4}, which is an essential guiding principal in the
definitions.
By symmetry we have, in the setting of \eqref{e:caZ},

\begin{equation}
\label{e:PiZPi}
\coprod_{i=1}^n\Pi(G_i(\R),\lambda)
\overset{1-1}\longleftrightarrow
\caZ
\overset{1-1}\longleftrightarrow
\coprod_{j=1}^m\Pi(\chG_j(\R),\ch\lambda)
\end{equation}
(here $\ch\lambda$ is a regular integral infinitesimal character for
$\ch G(\C)$).
See Corollary \ref{c:duality} for the  general statement, without the
restrictions on $\GC$.
In fact  this is a refinement of Vogan duality of \cite{ic4}, which considers
only a single real form  at a time. See \cite[Theorems 1.24 and 15.12]{abv}. 

This explains the nature of R-packets: it is clear
from the discussion that the Vogan dual of an R-packet for $G(\R)$
is an L-packet for some real form  of $\ch G(\C)$.

\bigskip
For ease of exposition in this Overview we have made various assumptions.
Removing these constraints involves a number of closely related
technical issues:

\begin{enumerate}
\item 
If $\GC$ is not adjoint we weaken the assumption $x^2=1$ to $x^2\in Z(\GC)$.
This introduces the notion of {\it strong
  real form}
(Section \ref{s:extended}).
\item
If (the derived group of) $\GC$ is not simply connected there are a
finite number of different infinitesimal characters at which the
representation theory looks different. Dually, since $\chG(\C)$ is not necessarily
adjoint, we allow $y^2\in Z(\chG(\C))$.
\item
If $\rho$ does not exponentiate to $\GC$ we  allow characters of
the $\rho$-cover of $H(\R)$ \cite[Definition 8.11]{av1}.
\item
If $\GC$ admits non-trivial outer automorphisms, we specify an
involution $\gamma$ in  $\Out(\GC)$. Then all of the objects discussed
above become ``twisted'' by $\gamma$.
\end{enumerate}

Here is an example which takes some of these issues into account. We
change notation to be more in agreement with the rest of the paper.

Let $G=Sp(2n,\C)$ and let $\chG$ be the dual group, i.e. $SO(2n+1,\C)$.
Fix Cartan subgroups $H$ of $G$ and $\ch H$ of $\chG$, with 
Lie algebras $\h$ and $\ch\h$.

\begin{theoremplain}
The irreducible representations of $Sp(2n,\R)$ with the same infinitesimal
character as the trivial representation are parametrized by:

$$
\{(x,y)\}/H\times\ch H
$$
\end{theoremplain}
where
\begin{equation}
\begin{aligned}
x&\in \Norm_G(H),\quad x^2=-I,\\
y&\in\Norm_{\chG}(\ch H)\quad y^2=I,
\end{aligned}
\end{equation}
and
\begin{equation}
(\Ad(x)|_{\h})^t=-Ad(y)|_{\ch\h}.
\end{equation}
The quotient is by the conjugation action of $H\times\ch H$.
This is a finite set. The number of elements is given by the following
table:

\bigskip
\begin{tabular}{|c|c|c|c|c|c|c|c|c|c|}
\hline
n&1&2&3&4&5&6&7&8&9\\
\hline
&4&18&88&460&2,544&1,4776&89,632&565,392&3,695,680\\
\hline
\end{tabular}

\medskip

This set also parametrizes the irreducible representations of real
forms  $SO(p,q)$  of $SO(2n+1,\C)$ with trivial infinitesimal character.
This is an example of Vogan duality (see \eqref{e:PiZPi}).

If we instead require $x^2=I$ this parametrizes representations of
the groups $Sp(p,q)$. Dually we  obtain representations of
$SO(p,q)$ with infinitesimal character $2\rho$.
\bigskip

We conclude this section with a comment about the requirements
that the infinitesimal character $\lambda$ be regular and integral. 
It is straightforward to remove both of these conditions, 
but introduces some extra complications.

First of all if $\lambda$ is regular
but not integral we replace $\ch G(\C)$ with a subgroup whose root
system is dual to the integral root system defined by $\lambda$. This
follows the program of \cite{ic4}, and with this change many of the
preceding constructions hold. 

Secondly if $\lambda$  is singular (and possibly non-integral) 
let $\lambda'$ be a regular element such that
$\lambda-\lambda'$ is a sum of roots.
By Zuckerman's translation principle \cite{translation_principle}
$\Pi(\GR,\lambda)$ is a an explicitly computable 
subset of $\Pi(\GR,\lambda')$, thereby reducing this case to the one of regular
infinitesimal character.

\subsection{Outline}
\label{s:outline}

Here is an outline of the contents of the paper.

Section \ref{s:groups} defines the basic objects of study, i.e. reductive groups and
root data.
In Section \ref{s:involutions} we discuss real forms of a complex group.

We put the information from Sections \ref{s:groups} and \ref{s:involutions} together to define {\it
  basic data} in Section \ref{s:basic}. This consists of either a pair
$(G,\gamma)$ consisting of a complex reductive group and an involution
in $\Out(G)$, or a pair $(D_b,\gamma)$ consisting of a based root
datum and an involution of it.

Given basic data we define the {\it extended group}
$G^\Gamma=G\rtimes\Gamma$ in Section \ref{s:extended}.
We also define the notions of {\it strong involution} and
{\it strong real form}.
Harish-Chandra modules for strong real forms are discussed in Section
\ref{s:representations}.

The first main step in constructing the parameter space is L-data
(Section \ref{s:ldata}).
The relation with $K$-orbits on $G/B$ is the  subject of Section
\ref{s:kgb}.

The primary combinatorial construction is the {\it one-sided parameter
  space} $\X$ of Section \ref{s:onesided}. Once we have $\X$ it is
straightforward to define the parameter space
$\caZ\subset\X\times\ch\X$. The main result is Theorem \ref{t:main}.

After stating the main result, we go back down into some of the
details of the space $\X$ in Sections \ref{s:fibers}-\ref{s:cayley} and relate this space to
structure theory of $G$. A summary of the relationship between $\X$
and structure theory of $G$ is found in Section \ref{s:recap}.
We discuss the {\it reduced parameter space} $\X^r$ in Section \ref{s:reduced}.
In Section \ref{s:tits} we use the Tits group
to explicitly compute the space $\X$.

Some examples are discussed in the body of the paper. In particular
the very informative cases of $SL(2)$ and $PSL(2)$ are discussed in
examples
\ref{ex:sl2C},
\ref{ex:sl2_1},
\ref{ex:sl2_2},
and \ref{ex:psl2_1}.
 More examples
may be found in \cite{adams_snowbird}.
\sec{Reductive Groups and Root Data}
\label{s:groups}

For many
purposes one may identify a connected reductive algebraic 
group with its group of complex points.
For the discussion
of real forms (Section \ref{s:involutions}), and to keep
the exposition as elementary as possible, we choose to work with complex
groups. Experts, and those with an interest in other fields, may wish
to convert to the language of algebraic groups where appropriate.

We now describe the parameters for a connected reductive complex
group.  These are provided by {\it root data} and {\it based root
 data}. Good references are the books by Humphreys
\cite{humphreys_lag} and Springer \cite{springer_book}.

We begin with a pair $X,\ch X$ of 
free abelian groups of finite rank, together with 
a perfect pairing $\langle\,,\,\rangle:X\times\ch X\rightarrow \Z$.
Suppose $\Delta\subset X,\ch\Delta\subset \ch X$ are finite sets,
equipped with a bijection $\alpha\rightarrow\ch\alpha$.
For $\alpha\in \Delta$ define the reflection
$s_\alpha\in\Hom(X,X)$:
$$
s_\alpha(x)=x-\langle x,\ch\alpha\rangle\alpha \quad (x\in X)
$$
and define $s_{\ch\alpha}\in \Hom(\ch X,\ch X)$ similarly.

A {\it root datum} is a quadruple
\begin{equation}
\label{e:rootdatum}
D=(X,\Delta,\ch X,\ch\Delta)
\end{equation}
where $X,\ch X$ are free abelian groups of finite rank, in duality via
a perfect pairing $\langle\,,\,\rangle$, and
$\Delta,\ch\Delta$ are finite subsets of $X,\ch X$, respectively. 
We assume there is a bijection
$\Delta\ni\alpha\mapsto\ch\alpha\in\ch\Delta$ such that for
all $\alpha\in\Delta$,
\begin{equation}
\label{e:bijection}
\langle\alpha,\ch\alpha\rangle=2,\,
s_\alpha(\Delta)=\Delta,\, s_{\ch\alpha}(\ch\Delta)=\ch\Delta.
\end{equation}

Suppose we are given $X,\ch X$ and finite subsets  $\Delta\subset X$ and
$\ch\Delta\subset\ch X$. 
By \cite[Lemma VI.1.1]{bourbaki_root_systems} applied to $\Q\langle\Delta\rangle$
and $\Q\langle\ch\Delta\rangle$ there is at most one bijection
$\alpha\mapsto\ch\alpha$ satisfying \eqref{e:bijection}.
Alternatively suppose we are given only a finite subset $\Delta$ of $X$,
satisfying $X\subset\Q\langle\Delta\rangle$. 
By (loc.~cit.) there is at most one subset $\ch\Delta$, and bijection
$\alpha\mapsto\ch\alpha$, satisfying \eqref{e:bijection}.
The condition  $X\subset \Q\langle\Delta\rangle$ holds if and only if the corresponding group is
semisimple.

Suppose $D_i=(X_i,\Delta_i,\ch X_i,\ch\Delta_i)$ ($i=1,2$) are root
data. We say they are isomorphic if there is an isomorphism
$\phi\in\Hom(X_1,X_2)$  satisfying 
$\phi(\Delta_1)=\Delta_2$ and 
$\phi^t(\ch\Delta_2)=\ch\Delta_1$.
Here $\phi^t\in \Hom(\ch X_2,\ch X_1)$ is defined by
\begin{equation}
\label{e:transpose}
\langle\phi(x_1),\ch x_2\rangle_2= \langle x_1,\phi^t(\ch
x_2)\rangle_1\quad (x_1\in X_1,\ch x_2\in X_2^\vee).
\end{equation}

Let $G$ be a connected reductive complex group and choose a
Cartan subgroup $H$ of $G$.
Let $X^*(H),X_*(H)$ be the character and co-character lattices of $H$
respectively.
We have canonical isomorphisms
\begin{equation}
\label{e:XH}
\h\simeq X_*(H)\otimes_\Z\C,\quad
\h^*\simeq X^*(H)\otimes_\Z\C
\end{equation}
where $\h$ is the Lie algebra of $H$ and $\h^*=\Hom(\h,\C)$.
(If $G$ has rank $n$ then $H\simeq (\C^\times)^n$,  $X_*(H)\simeq\Z^n$
and $\h\simeq\C^n$.) 
Let $\Delta=\Delta(G,H)$ be the set of roots of $H$ in $G$, and
$\ch\Delta=\ch\Delta(G,H)$ the corrresponding co-roots. 
Associated to $(G,H)$ is the root datum

\begin{equation}
\label{e:DGH}
D(G,H)=(X^*(H),\Delta,X_*(H),\ch\Delta).
\end{equation}

If $H'$ is another Cartan subgroup then there is an element $g\in G$
so that $gHg\inv=H'$.
Let $D'=(X^*(H'),\Delta(G,H'),X_*(H'),\ch\Delta(G,H'))$ be the
corresponding root datum.
The inverse transpose action on characters induces an isomorphism
$$
(Ad(g)^{t})\inv:X^*(H)\rightarrow X^*(H')
$$
which gives an isomorphism $D(G,H)\simeq D(G,H')$.

Now suppose in addition to $H$ we have chosen a Borel subgroup $B$
containing $H$. Let 
$\Delta^+$ be the corresponding set of positive roots of $\Delta$, with
simple roots $\Pi$.
Then $\Pi^{\vee}=\{\ch\alpha\,|\, \alpha\in\Pi\}$ is a set of
simple roots of $\ch\Delta$, and 

\begin{equation}
\label{e:DbGBH}
D_b(G,B,H)=(X,\Pi,\ch X,\Pi^{\vee})
\end{equation}
is a  {\it based root datum}.
If $H'\subset B'$ are another Cartan and Borel subgroup then
there is a {\it canonical} isomorphism  $D_b(G,B,H)\simeq D_b(G,B',H')$.

Each root datum is the root datum of a reductive
algebraic group, which is determined uniquely up to isomorphism, and
the same holds for based root data.

Note that a connected reductive complex group $G$ of rank $n$ is determined by
a small finite set of data: two sets (of order the semisimple rank of
$G$)
of integral $n$-vectors, subject only
to condition \eqref{e:bijection}, which may be expressed by 
requiring that the matrix of dot products is a Cartan matrix.

\begin{exampleplain}
Suppose $G$ is of rank $2$ and semisimple rank $1$.
Then a root datum is given by an ordered  pair  $(v,w)$ with $v,w\in \Z^2$,
satisfying $v\cdot~w=2$. Equivalence is given by the action of
$GL(2,\Z)$, where $g\cdot(v,w)=(gv,{}^tg\inv w)$.
It is an
interesting exercise to see that there are precisely three such pairs,
up to equivalence: $((2,0),(1,0)),((1,0),(2,0))$ and $((1,1),(1,1))$,
corresponding to $SL(2,\C)\times \C^\times, PGL(2,\C)\times \C^\times$
and $GL(2,\C)$,
respectively.

\end{exampleplain}

If $D=(X,\Delta,\ch X,\ch\Delta)$ is a root datum then the dual root
datum is defined to be $\ch D=(\ch X,\ch\Delta,X,\Delta)$. 
Given $G$ with root datum $D=(X,\Delta,\ch X,\ch\Delta)$
the {\it dual group} is the group $\ch G$ defined by $\ch D$.
We define 
duality of based root  data similarly.

\subsec{Automorphisms}
\label{s:automorphisms}
There is an exact sequence
\begin{equation}
\label{e:exact}
\exact {\Int(G)}{\Aut(G)}{\Out(G)}
\end{equation}
where $\Int(G)$ is the group of inner automorphisms of $G$,
$\Aut(G)$ is the group of (holomorphic)  automorphisms of $G$, 
and $\Out(G)\simeq\Aut(G)/\Int(G)$ is the group of outer
automorphisms.

A {\it splitting datum} or {\it pinning} for $G$ is a set
$(B,H,\{X_\alpha\})$
where $B$ is a Borel subgroup, $H$ is a Cartan subgroup contained in
$B$ and 
$\{X_\alpha\}$ is a set of  root vectors for the simple roots of $H$
in
$B$.

\begin{definition}
\label{d:distinguished}  
An involution of $G$ is said to be {\it distinguished } if it
preserves a  splitting datum.
\end{definition}
\begin{subequations}
\renewcommand{\theequation}{\theparentequation)(\alph{equation}}  

For example an inner involution is distinguished if and only if it is
the identity; this is the Cartan involution of  the compact real form.  More
generally, among all of the involutions mapping to a fixed involution
in $\Out(G)$, a distinguished involution 
makes as many roots as possible imaginary and
compact (cf.~Section \ref{s:waction}); it is the Cartan involution of a ``maximally compact''
real form of $G$ in the given inner class (cf.~Definition \ref{d:inner}).

The group $\Int(G)$ acts simply transitively on the set of splitting
data. Given a splitting datum $(B,H,\{X_\alpha\})$ this gives an isomorphism 
\begin{equation}
\label{e:phiS}
\phi_S:\Out(G)\simeq \Stab_{\Aut(G)}(S)\subset \Aut(G)
\end{equation}
and this is a splitting of the exact sequence
\eqref{e:exact}.
We call a  splitting  {\it distinguished} if 
it fixes a splitting datum.
We obtain isomorphisms
\begin{equation}
\label{e:out}
\Out(G)\simeq \Aut(D_b)\simeq \Aut(D)/W.
\end{equation}
If $G$ is semisimple then there is an injection of
$\Out(G)$ into the automorphism group of the Dynkin diagram of $G$; 
if $G$ is semisimple and simply connected or adjoint then this is an isomorphism.

\end{subequations}

If $\tau\in \Aut(D)$ then $-\tau^t\in \Aut(\ch D)$ 
(cf.~\ref{e:transpose}).
Now suppose $\tau\in \Aut(D_b)$. 
While  $-\tau^t$ is probably not 
in $\Aut(\ch D_b)$, if we let $w_0$ be the long element of the Weyl
group we have $-w_0\tau^t\in \Aut(\ch D_b)$.

\begin{definition}
\label{d:chtau}
Suppose $\tau\in \Aut(D_b)$. Let $\ch\tau=-w_0\tau^t\in \Aut(\ch
D_b)$. This defines a bijection  $\Aut(D_b)\overset{1-1}\longleftrightarrow \Aut(\ch D_b)$. 
By \eqref{e:out} we obtain a bijection $\Out(G)\leftrightarrow
\Out(\ch G)$ by composition:
\begin{equation}
\Out(G)\leftrightarrow\Aut(D_b)\leftrightarrow \Aut(\ch D_b)\leftrightarrow\Out(\ch G).
\end{equation}
For $\gamma\in\Out(G)$ we write $\ch\gamma$ for the corresponding
element of $\Out(\ch G)$. The map $\gamma\mapsto\ch\gamma$ is a
bijection of sets.
\end{definition}

\begin{remarkplain}
This is not necessarily an isomorphism of groups. For example the
identity goes to the image of $-w_0$ in $\Out(G)$, which is the
identity if and only $G$ is semisimple and $-1\in W$.
\end{remarkplain}

\begin{exampleplain}
Let $G=PGL(n)$ $(n\ge 3)$. Then $\chG=SL(n)$ and $\Out(G)\simeq
\Out(\chG)\simeq\Ztwo$. If $\gamma=1\in\Out(G)$ then $\ch\gamma$ is
the non-trivial element of $\Out(\chG)$.
It is represented by the automorphism $\ch\tau(g)={}^tg\inv$ of
$\chG$.
\end{exampleplain}

\sec{Involutions of Reductive Groups}
\label{s:involutions}

Fix a connected reductive complex group $G$. A real form of $G$ is a
subgroup $G(\R)$ which is the fixed points of an {\it antiholomorphic}
involution of $G$. Two such real forms are said to be equivalent if
they are conjugate by $G$.  The Cartan involution provides a
description of real forms in terms of {\it holomorphic} involutions
which is better suited to our purposes. We illustrate this with an
example. 

\begin{exampleplain}
Let $G=GL(n,\C)$. The  group $GL(n,\R)$ is a real form of
$GL(n,\C)$; it is the fixed points of the antiholomorphic
involution $\sigma(g)=\overline g$. The orthogonal group $O(n)$ is a
maximal compact subgroup of $GL(n,\R)$; it is the fixed points of the
involution $\theta(g)={}^tg\inv$ of $GL(n,\R)$. This involution
extends to a  holomorphic involution of $GL(n,\C)$, with fixed points $O(n,\C)$.
\end{exampleplain}

Suppose $\sigma$ is an antiholomorphic involution of $G$, with fixed
points $G(\R)$.
Then there is a holomorphic involution $\theta$ of $G$ such that
$G(\R)^\theta$ is a maximal compact subgroup of $G(\R)$.
It turns out this gives a bijection on the
level of $G$-conjugacy classes.
See \cite[Theorem 5.1.4]{ov}, \cite{av1} and 
\cite[Chapter X, Section 1]{helgason_book}.

\begin{theorem}
\label{t:realforms}
The map $\sigma\mapsto\theta$ gives a bijection between
$G$-conjugacy classes of antiholomorphic involutions of $G$ and 
$G$-conjugay classes of holomorphic involutions of $G$.
\end{theorem}

Using this result we classify real forms in terms of {\it holomorphic}
involutions. We also prefer to incorporate the notion of equivalence
in the definition of real forms. 

\begin{definition}
\label{d:involution}
Let $G$ be a connected reductive complex group.
An involution of $G$ is an involution in $\Aut(G)$, i.e. a holomorphic
automorphism $\theta$ of $G$ 
satisfying $\theta^2=1$.   A {\it real form} of $G$ is a $G$-conjugacy
class of involutions.
\end{definition}

Our definition of equivalence of real forms differs from the usual one
in one subtle way: it only allows
conjugacy by $G$, rather than all of $\Aut(G)$.
The theorem also holds with $G$-conjugacy 
replaced by conjugacy by $\Aut(G)$, and this is how it is
usually stated (for example \cite[Theorem 5.1.4]{ov}).

\begin{exampleplain}
\label{ex:sl2C}
We consider the case of $SL(2,\C)$.
Up to conjugacy there are two antiholomorphic
involutions of $SL(2,\C)$: $\sigma^s(g)=\overline g$ and $\sigma^c(g)={^t\overline
  g\inv}$. These are the two real forms  $G(\R)=SL(2,\R)$ (split) and
$G(\R)=SU(2)$ (compact) of $SL(2,\C)$, respectively. 

Equivalently $SL(2,\C)$ has two equivalence classes of holomorphic
involutions. Let $\theta^s(g)=tgt\inv$ where $t=\diag(i,-i)$.
Then $K(\C)=G^{\theta^s}=\C^\times$, so $K(\R)=S^1$, the maximal compact
subgroup of the corresponding real form $SL(2,\R)$. On the other hand let $\theta^c(g)=g$, so
$K(\C)=SL(2,\C)$, $K(\R)=SU(2)$, and the corresponding real form is $SU(2)$.
\end{exampleplain}

\begin{remarkplain}
\label{r:oldrealforms}
By Theorem \ref{t:realforms} there is a bijection between 
{\it equivalence classes of real forms} in the usual sense of
$G(\R)$ and antiholomorphic involutions, and {\it real forms} in the sense of
Definition \ref{d:involution}.
We work almost entirely with the latter notion; in the few
places we refer to the former the distinction will be
clear since we will refer to a real form $G(\R)$.
\end{remarkplain}

\begin{definition}
\label{d:inner}
An involution  $\theta\in \Aut(G)$ 
(Definition \ref{d:involution})
is in the inner class of $\gamma\in
\Out(G)$ if $\theta$ maps to $\gamma$ in the exact sequence
\eqref{e:exact}. 
If  $\theta,\theta'$ are involutions of $G$ we say $\theta$ is inner
to $\theta'$ if $\theta$ and $\theta'$ have the 
same image in $\Out(G)$.
\end{definition}

This corresponds to the usual notion of inner form \cite[12.3.7]{springer_book}.

\begin{remarkplain}
\label{r:realform}
The results of \cite{abv} are stated in terms of 
antiholomorphic, 
rather than holomorphic, involutions. See Remark
\ref{r:stronginvolutions}.
\end{remarkplain}

\begin{remarkplain}
As discussed at the beginning of Section \ref{s:groups}, this is the
one situation in which we need complex, as opposed to algebraic,
groups. Once we are in the setting of 
Cartan involutions, we could replace $G$ with the underlying algebraic group
$\mathbb G$, and holomorphic involutions of $G$ with {\it algebraic}
involutions of $\mathbb G$. 
This is the setting of \cite{richardson_springer}.
\end{remarkplain}

\sec{Basic Data}
\label{s:basic}

The setting for our mathematical questions will be:

\begin{enumerate}
\item A connected reductive complex group $G$,
\item An involution $\gamma\in\Out(G)$.
\end{enumerate}

We refer to $(G,\gamma)$ as {\it basic data}.
We say $(G,\gamma)$ is equivalent to $(G',\gamma')$ if there is an
isomorphism $\phi:G\rightarrow G'$ such that
$\gamma'\circ\phi=\phi\circ\gamma$.

On the other hand the software works entirely in the setting of based root
data. Given $(G,\gamma)$, let $(B,H,\{X_\alpha\})$ be a splitting
datum (Section  \ref{s:automorphisms}).
From these we obtain:

\begin{enumerate}
\item[(a)] A based root datum $D_b=D_b(G,B,H)$,
\item[(b)] An involution $\gamma$ of $D_b$  (cf.~\eqref{e:out}).
\end{enumerate}
The based root datum $D_b$ and its involution are independent of the
choice of splitting datum, up to canonical isomorphism. Conversely,
given a based root datum $D_b$ with an involution $\gamma$ we can
construct $(G,\gamma)$, uniquely determined up to isomorphism. 

Given basic data $(G,\gamma)$, choose a splitting datum
$(B,H,\{X_\alpha\})$. 
A key to our algorithm is that $(B,H,\{X_\alpha\})$ is fixed once and
for all.
This enables us to do all of our constructions on a fixed Cartan
subgroup. Let $W=W(G,H)$ be the Weyl group. 

The weight lattice for $G$ is
\begin{equation}
\label{e:P}
P=\{\lambda\in X^*(H)\otimes\C\,|\,
\langle\lambda,\ch\alpha\rangle\in\Z\text{ for all }\alpha\in\Delta\}
\end{equation}
and dually the co-weight lattice for $G$ is
\begin{subequations}
\renewcommand{\theequation}{\theparentequation)(\alph{equation}}  
\begin{equation}
\label{e:chP}
\ch P=\{\ch\lambda\in X_*(H)\otimes\C\,|\, \langle\alpha,\ch\lambda\rangle\in\Z\text{
for all }\alpha\in\Delta\}.
\end{equation}
These are actually lattices 
only if $G$ is semisimple.
We may identify $2\pi i X_*(H)$ with the kernel of $\exp:\h\rightarrow
H(\C)$. Under this identification
\begin{equation}
\label{e:under}
\ch P=\{\ch\lambda\in\h\,|\, \exp(2\pi i\ch\lambda)\in Z(G)\}.
\end{equation}
We write $P(G,H)$ and  $\ch P(G,H)$ to indicate the dependence on $G$
and $H$.

We also define
\begin{equation}
\label{e:Preg}
P_{\text{reg}}=\{\lambda\in P\,|\,
\langle\lambda,\ch\alpha\rangle\ne0\text{ for all }\alpha\in\Delta\}
\end{equation}
and
\begin{equation}
\label{e:chPreg}
\chPreg=\{\ch\lambda\in \ch P\,|\,
\langle\alpha,\ch\lambda\rangle\ne0\text{ for all }\alpha\in\Delta\}.
\end{equation}
\end{subequations}

By duality we obtain the dual based root datum $\ch D_b$, 
an involution  $\ch\gamma$ of $D_b^\vee$
(cf.~Definition \ref{d:chtau}),  the dual group $\ch G$, and 
thus {\it dual}  basic data $(\ch G,\ch \gamma)$.
In particular $\ch G$ comes equipped with fixed Cartan and Borel
subgroups $\ch H\subset \ch B$.
We have $X^*(H)=X=X_*(\ch H)$ and $X_*(H)=\ch X=X^*(\ch H)$. These are
canonical identifications.
By \eqref{e:XH} applied to $H$ and $\ch H$ these induce
identifications 
$\h=\h^{\vee*}$ and $\h^*=\h^{\vee}$.
Note that $P(G,H)=\ch P(\ch G,\ch H)$, and 
by \eqref{e:under} we have
\begin{equation}
\label{e:pgh}
P(G,H)=\{\lambda\in \ch\h\,|\, \exp(2\pi i\lambda)\in
Z(\ch G)\}.
\end{equation}

We are also given the bijection
$\Delta=\Delta(G,H)\ni\alpha\rightarrow \ch\alpha\in
\ch\Delta=\Delta(\ch G,\ch H)$.
We identify $W(G,H)$ with $W(\ch G,\ch H)$ by the map 
$W(G,H)\ni w\rightarrow w^t\in W(\ch G,\ch H)$
(cf.~\eqref{e:transpose});
equivalently, $s_\alpha\rightarrow s_{\ch\alpha}$. 

\sec{Extended Groups and Strong Real Forms}
\label{s:extended}
Fix basic data $(G,\gamma)$ as in Section \ref{s:basic}, and a
splitting datum $(H,B,\{X_\alpha\})$ (cf.~Section
\ref{s:automorphisms}). 
This gives us a
distinguished splitting of the exact sequence \eqref{e:exact}, taking $\gamma$ to a
distinguished involution of  $G$ (Definition \ref{d:distinguished})
which we also denote $\gamma$.  Let
$\Gamma=\{1,\sigma\}$ be the Galois group of $\C/\R$.

\begin{definition}
\label{d:extended}
The extended group for $(G,\gamma)$  is the semi-direct product
\begin{equation}
\GGamma=G\rtimes\Gamma
\end{equation}
where $\sigma\in\Gamma$ acts on $G$ by the distinguished involution
$\gamma$.

We say the element $1\times\sigma$ is distinguished, and denote it $\delta$.
\end{definition}

Recall $\gamma$ is a Cartan involution of a maximally compact form in this inner class.

\begin{exampleplain}
\label{e:equalrank}
Suppose $\gamma=1$, so $\GGamma=G\times\Gamma$.  Note that $\gamma$
is the Cartan involution of the compact real form of $G$, and 
this is the ``equal rank'' case. Every real form $\GR$ in this inner
class contains a compact Cartan subgroup,
and every  involution in this inner class is contained in $\Int(G)$.

In particular suppose $G$ is semisimple and the Dynkin diagram of $G$
has no non-trivial automorphisms. Then $\gamma$ is necessarily
trivial, and every involution of $G$ is inner.
\end{exampleplain}

\begin{exampleplain}
\label{e:slnc}  
Suppose $1\ne \gamma\in \Out(G)$.
The most convenient way to compute the corresponding distinguished
involution $\gamma$ of $G$ (cf.~Definition \ref{d:distinguished})
is to list the real forms in this inner class, and choose the most
compact one.

For example let $G=SL(n,\C)$ ($n\ge 3$). There is a unique outer
automorphism of the Dynkin diagram, corresponding to the 
non-trivial element $\gamma\in\Out(G)$.
If $n$ is odd the only
real form in this inner class is the split group $SL(n,\R)$, and we take
$\gamma(g)={^tg\inv}$, the Cartan involution of $SL(n,\R)$.

If $n$ is even then $G$ has two real forms, $SL(n,\R)$ and
$SL(n/2,\H)$ where $\H$ is the quaternion algebra.
We  let $\gamma$ act by a Cartan involution of
$SL(n/2,\H)$.
\end{exampleplain}

The  extended group $\GGamma$ encapsulates all of the real
forms of $G$ in the inner class defined by $\gamma$.
That is if $\xi\in \GGamma\bs G$ satisfies $\xi^2\in Z(G)$  then
$\int(\xi)$ is an involution in the inner class of $\gamma$.
Conversely, if $\theta$ is in the inner class of
$\gamma$, then there is an element $\xi\in\GGamma\bs G$
with $\xi^2\in Z(G)$ and $\theta=\int(\xi)$.
For the algorithm it is important to keep track of $\xi$, and not just
$\theta=\int(\xi)$.

\begin{definition}
\label{d:stronginvolution}
A strong involution of $G$ in the inner class of $\gamma$
is an element $\xi\in \GGamma\bs G$ such that
$\xi^2\in Z(G)$. 
The set of such strong involutions is denoted
$\I(G,\gamma)$.
We define a  {\it strong real form} of $G$ in the inner class of
$\gamma$ to be a $G$-conjugacy class
of strong involutions.

For $\xi\in\I(G,\gamma)$ let $\theta_\xi=\int(\xi)$ and
$K_\xi=\Stab_G(\xi)=G^{\theta_\xi}$. 
\end{definition}

If $\gamma$ is understood we refer to strong involutions and strong
real forms of $G$.

\begin{exampleplain}
\label{ex:sl2_1}
Recall (Example \ref{ex:sl2C}) $SL(2,\C)$ has two real forms,
$\theta^c=1$ (i.e. $SU(2)$) and 
$\theta^s=\int(\diag(i,-i))$ (i.e. $SL(2,\R)=SU(1,1))$. 
However 
$SL(2,\C)$ has {\it three} strong real forms, i.e. 
conjugacy classes of 
{\it strong} involutions:
$\xi=I,-I$ and $\diag(i,-i)$.
(Here and elsewhere, when $\gamma=1$, we write $\xi$ for the element
$(\xi,\sigma)\in G^\Gamma\bs G$.)
Then 
$\theta_\xi=\theta^c$ if $\xi=\pm I$, or $\theta_\xi=\theta^s$ if $\xi=\diag(i,-i)$. 
We can think of these strong real forms as $SU(2,0), SU(0,2)$ and
$SU(1,1)\simeq SL(2,\R)$, respectively.

Now consider $PSL(2,\C)\simeq SO(3,\C)$. This has two real forms: the
compact group $SO(3,\R)$ (with $K(\R)=SO(3,\R)$, $K(\C)=SO(3,\C)$) and
the split one $SO(2,1)$ (with $K(\R)=O(2)$, $K(\C)=O(2,\C)$).
Up to conjugacy there are two strong involutions: $I$ and
$\diag(-1,-1,1)$.
\end{exampleplain}

The next Lemma is immediate from the definitions 
(cf.~Definition \ref{d:inner}).

\begin{lemma}
We have
$$
\I(G,\gamma)/G=
\{\text{strong real forms in the inner class of }\gamma\}.
$$
The map $\I(G,\gamma)\ni \xi\mapsto \theta_\xi$
is a surjection from $\I(G,\gamma)$ to the set of involutions
in the inner class of $\gamma$. It factors to a surjection
$$
\I(G,\gamma)/G\twoheadrightarrow
\{\text{real forms of $G$ in the inner class of $\gamma$}\}.
$$
If $G$ is adjoint this is a bijection.
\end{lemma}

\begin{exampleplain}
In Example \ref{ex:sl2_1} the map from strong real forms to real
forms is bijective for  $PSL(2)$.
If $G=SL(2)$ the fiber of the map is a single element for the split
real form $SU(1,1)$, and two elements ($SU(2,0)$ and $SU(0,2)$, so to
speak) for the compact real form $SU(2)$.
\end{exampleplain}

\begin{exampleplain}
\label{ex:equalrank}
This is generalization of  Example \ref{ex:sl2_1}.
Suppose $\gamma$ is the identity.
The  strong real forms of $G$ are parametrized by 
\begin{subequations}
\renewcommand{\theequation}{\theparentequation)(\alph{equation}}  
\begin{equation}
\{\xi\in G\,|\,\xi^2\in Z(G)\}/G.
\end{equation}
(As in Example \ref{ex:sl2_1} since $\gamma=1$ we write $\xi$
instead of $(\xi,\sigma)$.)
Every strong involution $\xi$ is conjugate to an element of $H$ so
this set is the same as
\begin{equation}
\{\xi \in H\,|\, \xi^2\in Z(G)\}/W
=(\frac12\ch P/X_*(H))/W
\end{equation}
where $\ch P$ is the coweight lattice  
(cf.~\eqref{e:chP}).
In particular  if $G$ is adjoint the real
forms of $G$ are parametrized by 
\begin{equation}
(\frac12\ch P/\ch P)/W.
\end{equation}
\end{subequations}
\end{exampleplain}

\begin{exampleplain}
Continuing with the preceding example, let $G=PSp(2n)$. 
In the usual coordinates $\ch P\simeq\Z^n\cup(\Z+\frac12)^n$, and for 
representatives of 
$(\frac12\ch P/\ch P)/W$ we may take
\begin{equation}
\frac14(1,\dots, 1)
\end{equation}
corresponding to $PSp(2n,\R)$ (the split group) and
\begin{equation}
\frac12(\overbrace{1,\dots,1}^p,\overbrace{0,\dots, 0}^q)\quad (p\le [n/2])
\end{equation}
corresponding to $PSp(p,q)$. 

For $G=Sp(2n)$ we have $X_*(H)=\ch R\simeq\Z^n$, and the strong real forms
are parametrized by 
\begin{equation}
(\frac12[\Z^n\cup(\Z+\frac12)^n]/\Z^n)/W.
\end{equation}
For representatives we may take $\frac14(1,\dots, 1)$ 
and 
$\frac12(\overbrace{1,\dots,1}^p,\overbrace{0,\dots, 0}^q)$ as before,
where now $0\le p\le n$. Thus there are two strong real forms mapping to each
real form $Sp(p,q)$ for $p\ne q$.

See \cite[Example 1-17]{bowdoin}.
  
\end{exampleplain}

\begin{subequations}
\renewcommand{\theequation}{\theparentequation)(\alph{equation}}  
\label{e:I}
We will make frequent use of the following construction.
Choose a set of representatives  $\{\xi_i\,|\, i\in I\}$ of the set of
strong real forms. That is 
\begin{equation}
\label{e:I1}
\{\xi_i\,|\,i\in I\}\simeq\I(G,\gamma)/G.
\end{equation}
If $G$ is semisimple (in fact if $Z(G)^\Gamma$ is finite,
cf.~\eqref{p:fibers}(3))
this
is a finite set.
For $i\in I$ let 
\begin{equation}
\label{e:I2}
\theta_i=\int(\xi_i),\quad K_i=G^{\theta_i}
\end{equation}
and
\begin{equation}
\I_i=\{\xi\in \I(G,\gamma)\,|\,\xi \text{ is $G$-conjugate to }\xi_i\}.
\end{equation}
The stabilizer of $\xi_i$ in $G$ is $K_i$, so $\I_i\simeq G/K_i$, and
we have
\begin{equation}
\label{e:IG}
\I(G,\gamma)\simeq\coprod_{i\in I}G/K_i.
\end{equation}

\end{subequations}

Recall the distinguished involution $\theta_\delta=\int(\delta)$
is the Cartan involution of a maximally compact real
form in the inner class of $\gamma$.
It is helpful to also invoke the most split form in this inner class.
We say an involution of $G$ is {\it quasisplit} if it is the Cartan
involution of a quasisplit real form of $G$. For  a characterization of quasisplit
involutions see \cite[Proposition 6.24]{av1} (where they are called
{\it principal}).
By \cite[Theorem 6.14]{av1} there is a unique conjugacy class of quasisplit involutions
in each inner class. 

We emphasize the symmetry of the situation by summarizing this:

\begin{lemma}
\label{l:dandq}
Let  $\GGamma$ be the  extended group for $(G,\gamma)$.
\hfil\break
\noindent (1) There exists a strong involution $\xi\in \GGamma$ so that
$\theta_\xi$ is distinguished. 
The involution $\theta_\xi$ is unique up to
conjugation by $G$.

\noindent (2)
There exists a strong involution $\eta$
so that $\theta_\eta$ is quasisplit.
The involution $\theta_\eta$ is unique up to
conjugation by $G$.
\end{lemma}

\begin{remarkplain}
\label{r:lgroup}
The extended group $\GGamma$  in \cite[Definition 9.6]{av1} is
defined in terms of a quasisplit involution, rather than a
distinguished one. The equivalence of the two definitions is the
content of \cite[9.7]{av1}.  
This discussion also shows that, applied to $(\chG,\ch\gamma)$, the
group $\chGGamma$ is isomorphic to the L-group of the real forms of
$G$ in the iner class of $\gamma$.
\end{remarkplain}

\begin{remarkplain}
\label{r:stronginvolutions}
Since in \cite{abv} we work with antiholomorphic involutions instead
of holomorphic ones,
(cf.~Remark \ref{r:realform}), 
the extended group $\GGamma$  in \cite[Chapter 3]{abv} is defined in
terms of an antiholomorphic involution.
The results are equivalent, but some translation is necessary between
the two pictures.
\end{remarkplain}

\sec{Harish-Chandra modules for strong real forms}
\label{s:representations}
Fix basic data $(G,\gamma)$ as in Section \ref{s:groups} and let
$\I=\I(G,\gamma)$ be the corresponding set of strong involutions.
For $\xi\in\I$ let $K_\xi=\Stab_G(\xi)$ as usual, and define
$(\g,K_\xi)$ modules as in \cite{green}.

\begin{definition}
\label{d:hc}
A Harish-Chandra module for  a strong involution is a pair 
$(\xi,\pi)$ where $\xi\in\I$ and $\pi$ is a $(\g,K_\xi)$-module.
A {\it Harish-Chandra module for a strong real form of $G$} is
the $G$-orbit of a pair $(\xi,\pi)$, where 
$g\cdot(\xi,\pi)=(g\xi g\inv,\pi^g)$. 
\end{definition}

An infinitesimal
character for $\g$ may be identified, via the Harish-Chandra
homomorphism with an orbit of $W$ on $\h^*$.
For $\lambda\in \h^*$ we
write $\chi_\lambda$ for the corresponding infinitesimal character;
note that $\chi_\lambda=\chi_{w\lambda}$ for all $w\in W$.
If $\pi$ is a Harish-Chandra module with infinitesimal character
$\chi_\lambda$ we may simply say $\pi$ has infinitesimal character $\lambda$.

\begin{definition}
\label{d:reginfchar}
We say $\lambda$ and $\chi_\lambda$ are  integral if $\lambda\in P$,
and regular and integral if $\lambda\in P_{\text{reg}}$ (cf.~\eqref{e:Preg}).
\end{definition}

\begin{definition}
\label{d:pig}
Given $\xi\in \I$ let $\Pi(G,\xi)$ be the set of equivalence classes of  irreducible $(\g,K_\xi)$-modules
with regular integral infinitesimal character.
Let 
\begin{equation}
\Pi(G,\gamma)=\{(\xi,\pi)\,|\,\xi \in \I,
\pi\in\Pi(G,\xi)\}/G.
\end{equation}
If $\Lambda$ is a subset of $P_{reg}$ let
$\Pi(G,\gamma,\Lambda)\subset \Pi(G,\gamma)$ be the set of
(equivalence classes of) pairs $(\xi,\pi)$ for which the infinitesimal
character of $\pi$ is an element of $\Lambda$.
\end{definition}

If we fix a set of representatives $I$ of $\I/G$ as in \eqref{e:I} we have
\begin{equation}
\label{e:PiG}
\Pi(G,\gamma)\simeq\coprod_{i\in I}\Pi(G,\xi_i).
\end{equation}
Thus $\Pi(G,\gamma)$ parametrizes  
Harish-Chandra modules for strong real forms of $(G,\gamma)$, with
regular integral infinitesimal character.
With the obvious notation we also have
$\Pi(G,\gamma,\Lambda)=\coprod_{i\in I}\Pi(G,\xi_i,\Lambda)$.

Fix $\xi\in \I$.  As a consequence of \cite{langlands_mimeo}, associated to each
$\chG$-conjugacy class of admissible homomorphisms
$\phi:W_\R\rightarrow \chGGamma$ is an L-packet
$\Pi_\phi(G,\xi)$ (see Section \ref{s:ldata}).
This is a finite set of (equivalence classes of)
$(\g,K_{\xi})$ modules, all having the same infinitesimal character. 
Each L-packet is finite, and non-empty if $\xi$ is quasisplit.  The
non-empty L-packets partition the set of  irreducible Harish-Chandra
modules for the strong involution $\xi$.

We define the {\it large} L-packet of $\phi$ to be the union, over all strong
involutions, of L-packets, modulo our notion of equivalence:
\begin{subequations}
\label{e:large}
\renewcommand{\theequation}{\theparentequation)(\alph{equation}}  
\begin{equation}
\Pi_\phi(G,\gamma)=\{(\xi,\pi)\,|\, \xi\in\I,\pi\in \Pi_\phi(G,\xi)\}/G.
\end{equation}
With $I$ as in \eqref{e:I} we have
\begin{equation}
\Pi_\phi(G,\gamma)\simeq \coprod_{i\in I}\Pi_\phi(G,\xi_i)
\end{equation}
and each set $\Pi_\phi(G,\xi_i)$ is finite.
\end{subequations}

\sec{L-data}
\label{s:ldata}
Fix basic data $(G,\gamma)$, with corresponding 
dual data $(\chG,\ch\gamma)$ (cf.~Section \ref{s:basic}).
Let $\chGGamma$ be the corresponding extended group
(Definition \ref{d:extended}).
Recall (Remark \ref{r:lgroup}) that $\chGGamma$ is isomorphic to the L-group of $G$.
We begin by parametrizing admissible homomorphisms of the Weil group into
$\chGGamma$
\cite{borel_corvallis}.

Let $\weil$ be the Weil group of $\R$. That is
$\weil=\langle\C^\times,j\rangle$ where $jzj\inv=\overline z$ and $j^2=-1$. 
An admissible homomorphism $\phi:\weil\rightarrow \chGGamma$ is a
continuous homomorphism such that $\phi(\C^\times)$ consists of semisimple
elements and $\phi(j)\in \chGGamma\bs \chG$. 
 
Suppose $\phi:\weil\rightarrow \chGGamma$ is an admissible
homomorphism. 
Then $\phi(\C^\times)$ is contained in a Cartan subgroup $\ch H_1(\C)$,
and $\phi(e^z)=\exp(2\pi i(z\lambda+\overline z\nu))$ for some
$\lambda,\nu\in\ch\h_1$.
Choose an inner isomorphism $\ch\h_1\simeq\ch\h$.
As explained in Section \ref{s:basic} we have $\ch\h=\h^*$,
so we may identify $\lambda$ with an element
(which we still call $\lambda$) of $\h^*$, whose $W$
orbit is well-defined.
We define the {\it infinitesimal character} of $\phi$ to be the
$W$-orbit of $\lambda\in\h^*$.

Recall  $P(G,H)=\{\lambda\in\h^*\,|\, \exp(2\pi i\lambda)\in
Z(\ch G)\}$ (cf.~\eqref{e:pgh}).

\begin{definition}
\label{d:onesidedcompleteldatum}
A complete one-sided L-datum for $(\chG,\ch\gamma)$ is a triple $(\eta,\ch B_1,\lambda)$ where  
$\eta$ is a strong involution of $\chG$ (Definition \ref{d:stronginvolution}),
$\ch B_1$ is a Borel subgroup of $\chG$ and  $\lambda\in P(G,H)$ 
satisfies $\exp(2\pi i\lambda)=\eta^2$.
The group $\chG$ acts by conjugation on this data, and let
\begin{equation}
\label{e:caPc}
\caP_c(\chG,\ch\gamma)=\{\text{complete one-sided L-data}\}/\chG.
\end{equation}
\end{definition}

Fix a complete one-sided L-datum
$\ch S_c=(\eta,\ch B_1,\lambda)$.
By  \cite{matsuki} (also see \cite[Lemma 6.18]{abv}) there is a 
$\theta_\eta$-stable Cartan subgroup $\ch H_1$ of $\ch B_1$, unique up to
conjugacy by $\ch K_\eta\cap\ch B_1$.
Choose $g\in\chG$ such that $g\ch Hg\inv=\ch H_1$ and
$\langle\Ad(g)\lambda,\ch\alpha\rangle\ge0$ for all $\alpha\in
\Delta(\ch B_1,\ch H_1)$. Let $\lambda_1=\Ad(g)\lambda\in \ch\h_1$.
Define $\phi:\weil\rightarrow\chGGamma$ by:
\begin{equation}
\label{e:phizj}
\begin{aligned}
\phi(z)&=z^{\lambda_1}\overline z^{\eta\lambda_1}\quad(z\in\C^\times)\\
\phi(j)&=\exp(-\pi i\lambda_1)\eta.
\end{aligned}
\end{equation}
The first statement is shorthand for
$\phi(e^z)=\exp(z\lambda_1+\Ad(\eta)\overline z\lambda_1)$.
It is easy to see $\phi$ is an admissible homomorphism, and the $\chG$
conjugacy class of $\phi$ is independent of the choices of $H_1$ and
$g$.
The next result follows readily. 

\begin{proposition}
\label{p:weil}
The map taking $(\eta,\ch B_1,\lambda)$ to $\phi$ defined by
\eqref{e:phizj} induces 
a bijection between $\caP_c(\chG,\ch\gamma)$ and the set of 
$\chG$-conjugacy classes of  admissible
homomorphisms $\weil\rightarrow \chGGamma$.
\end{proposition}

Recall (cf.~\ref{e:large}) associated to an admissible
homomorphism $\phi$ is a large L-packet $\Pi_\phi(G,\gamma)$.
For $\ch S_c\in\caP_c(\chG,\ch\gamma)$ define $\phi$ by
\eqref{e:phizj} and let
\begin{equation}
\label{e:large2}
\Pi_{\ch S_c}(G,\gamma)=\Pi_{\phi}(G,\gamma).
\end{equation}

Fix $(\eta,\ch B_1)$.
If
$\lambda\in\Preg$, $\ch S_c=(\eta,\ch B_1,\lambda)$ is a complete one-sided L-datum if
and only if $\exp(2\pi i\lambda)=\eta^2$.
For $\lambda$ satisfying this condition we may define the large
L-packet
$\Pi_{\ch S_c}(G,\gamma)$. Any two such L-packets are related by a {\it
translation functor} \cite{translation_principle}, and in this
sense the choice of $\lambda$ is not important.
We therefore drop the parameter $\lambda$ from complete L-data:

\begin{definition}
\label{d:onesidedldatum}
A one-sided L-datum for $(\chG,\ch\gamma)$ is a pair $\ch S=(\eta,\ch B_1)$ where  
$\eta$ is a strong involution of $\chG$ (Definition
\ref{d:stronginvolution}) and
$\ch B_1$ is a Borel subgroup of $\chG$.
Let
\begin{equation}
\label{e:caP}
\caP(\chG,\ch\gamma)=\{\text{one-sided L-data}\}/\chG.
\end{equation}
\end{definition}

There is a well defined map
\begin{equation}
\label{e:z(S)}
\caP(\chG,\ch\gamma)\mapsto Z(\chG)
\end{equation}
taking (the equivalence class of) $(\eta,\ch B_1)$ to $\eta^2$.
\medskip

For the purposes of  Proposition \ref{p:weil}
we have defined one-sided L-data $\caP(\chG,\ch\gamma)$ for
$(\chG,\ch\gamma)$. It is evident that the definition is symmetric, and
applies equally to $(G,\gamma)$. As discussed
in Section \ref{s:overview:Z} symmetrizing will give us the finer data
which parametrizes individual representations, instead of L-packets. 
So let 
\begin{equation}
\label{e:caPchP}
\begin{aligned}
\caP&=\caP(G,\gamma)\\
\ch\caP&=\caP(\chG,\ch\gamma).\\
\end{aligned}
\end{equation}
Define $\caP_c$ and $\ch\caP_c$ similarly.

\begin{definition}
\label{d:L}
An L-datum for $(G,\gamma)$ is a quadruple 
\begin{equation}
{\bf S}=(\xi,B_1,\eta,\ch B_1)
\end{equation}
where $S=(\xi,B_1)$ is an L-datum for $G$, and $\ch S=(\eta,\ch B_1)$ is an
L-datum for $\ch G$.
Let $\theta_{\xi,\h}$  (respectively $\theta_{\eta,\ch\h}$) be
$\theta_\xi$ restricted to $\h$ (resp. $\theta_\eta$ restricted to
$\ch\h$).
Recalling \eqref{e:transpose}, we require that these satisfy
\begin{equation}
(\theta_{\xi,\h})^t=-\theta_{\eta,{\ch\h}}.
\end{equation}

A complete L-datum is a set
\begin{equation}
\label{e:completeL}
{\bf S}_c=(\xi,B_1,\eta,\ch B_1,\lambda)
\end{equation}
where the same conditions hold,  $\lambda\in\Preg$ and
$\exp(2\pi i\lambda)=\eta^2$.

Let
\begin{equation}
\label{e:L}
\begin{aligned}
\caL&=\{\text{L-data}\}/G\times\chG\subset\caP\times\ch\caP\\
\caL_c&=\{\text{complete L-data}\}/G\times\chG
\subset \caP\times\ch\caP_c.
\end{aligned}
\end{equation}
\end{definition}

Suppose ${\bf S}_c=(S,\ch S_c)=(\xi,B_1,\eta,\ch B_1,\lambda)$ is a complete L-datum for $(G,\gamma)$. 
By \cite[Theorem 2.12]{bowdoin} associated to ${\bf S}_c$ is a
$(\g,K_\xi)$-module $I({\bf S}_c)$. 
This is a standard module, with regular integral infinitesimal character
$\lambda$, and has a unique irreducible quotient $J({\bf S}_c)$.
We obtain the following version of the Langlands classification.
For more detail, in particular  information on how to write $J({\bf S}_c)$ in
various classifications, see \cite[Section 9]{av1}.

\begin{theorem}[\cite{bowdoin}, Theorem 2-12]
\label{t:reps}
The map 
\begin{equation}
\label{e:caL}
\caL_{\text{c}}\ni {\bf S}_c\mapsto J({\bf S}_c)\in \Pi(G,\gamma)
\end{equation}
is a bijection.
\end{theorem}

The large L-packet $\Pi_{S_c}$ of \eqref{e:large2} is the set of
irreducible representations defined by L-data with second coordinate
$\ch S_c$:

$$
\Pi_{\ch S_c}=\{J({\bf S})\,|\, {\bf S}=(S,\ch S_c)\}.
$$

Recall the right hand side of \eqref{e:caL} consists of the equivalence classes of
pairs $(\xi,\pi)$ where $\xi$ is a strong involution and $\pi$ is a
$(\g,K_\xi)$-module. This is a somewhat subtle space. 
Furthermore, as discussed after  Proposition \ref{p:weil}, we would like
to replace $\caL_c$ with $\caL$ on the left side of \ref{e:caL}. 
This amounts to considering $\lambda$ only up to $X^*(H)$, and using
the translation principle.
With these considerations in mind 
we give  several alternative formulations of Theorem \ref{t:reps}.

Suppose $\GR$ is a real form of $G$ (cf.~Remark \ref{r:oldrealforms})
and $\Lambda\subset\Preg$.
Recall (cf.~\ref{d:infchar}) $\Pi(G(\R),\lambda)$ is the set of equivalence
classes of irreducible admissible representations with infinitesimal
character $\lambda$. 
By analogy with Definition \ref{d:pig} 
define $\Pi(\GR,\Lambda)=\coprod_{\lambda\in\Lambda}\Pi(\GR,\lambda)$.

\begin{theorem}
\label{t:ldata}
\hfil\break

\noindent (1) Fix a set $\Lambda\subset P_{\text{reg}}$ of representatives of
 $P/X^*(H)$. This is a finite set if $G$ is semisimple.
 The map \eqref{e:caL}
 induces a bijection 
\begin{equation}
\caL\overset{1-1}\longleftrightarrow \Pi(G,\gamma,\Lambda)
\end{equation}
(cf.~Definition \ref{d:pig}).
If $G$ is semisimple and simply connected we may take
$\Lambda=\{\rho\}$, the infinitesimal character of the trivial
representation.

\medskip\noindent (2) Let $I\simeq \I/G$ be a set of representatives of the
  strong real forms as in \eqref{e:I}. For
  each $i\in I$ let $G_i(\R)$ be the real form of $G$ corresponding to
  the strong involution $\xi_i$ (cf.~Section \ref{s:involutions}).
Choose $\Lambda$ as in (1).  The map  \eqref{e:caL} induces a bijection
\begin{equation}
\caL\overset{1-1}\longleftrightarrow \coprod_{i\in I}\Pi(G_i(\R),\Lambda).
\end{equation}

\medskip\noindent (3) Suppose $G$ is adjoint.
Write $G_1(\R),\dots, G_m(\R)$ for the real
forms of $G$ in the given inner class, and choose representatives
$\lambda_1,\dots,\lambda_n\in \Preg$
for $P/R$.
The map \eqref{e:caL}
induces a bijection:
\begin{equation}
\caL\overset{1-1}\longleftrightarrow \coprod_{i,j}\Pi(G_i(\R),\lambda_j).
\end{equation}
\end{theorem}

Sections \ref{s:onesided} and \ref{s:Z} will be concerned with finding a
combinatorial description of the set $\caL$. Most of the work involves
the one-sided parameter space $\caP(G,\gamma)$ \eqref{e:caPchP}. Before turning to this
we give a geometric interpretation of this space.

\sec{Relation with the flag variety}
\label{s:kgb}
\begin{subequations}
\renewcommand{\theequation}{\theparentequation)(\alph{equation}}  
Recall 
(Definition \ref{d:onesidedldatum}) the one-sided parameter space 
$\caP=\caP(G,\gamma)$ is the set of conjugacy classes of  
one-sided L-data, i.e.
pairs $(\xi,B_1)$ where $\xi$ is a
strong involution
(Definition \ref{d:stronginvolution})
and $B_1$ is a Borel subgroup.
The space $\caP$
has a natural interpretation in terms of
the flag variety.
We see this by conjugating any pair $(\xi,B_1)$ to one with $\xi$ in a
fixed set of representatives as in 
in  \eqref{e:I}.  In the next section we will instead conjugate $B_1$
to $B$, and thereby obtain a combinatorial model of $\caP$.

Let $\caB$ be the set of Borel subgroups of $G$.
Then the set of one-sided L-data for $(G,\gamma)$ is $\I\times\caB$,
and 
\begin{equation}
\label{e:caPIBG}
\caP=(\I\times \caB)/G.
\end{equation}
Every Borel subgroup is conjugate to $B$, and $\caB\simeq G/B$.
For $\xi\in \I$ let
\begin{equation}
\label{e:caPxi}
\caP[\xi]=\{(\xi',B_1)\,|\,\xi'\text{ is conjugate to }\xi\}/G.
\end{equation}
Then 
\begin{equation}
\caP[\xi]\simeq (G/K_\xi\times G/B)/G
\end{equation}
with $G$ acting by left multiplication.
It is  an elementary exercise to see the map 
\begin{equation}
(\xi',B_1)=(g\xi g\inv,hBh\inv)\mapsto K_\xi(g\inv h)B\quad (g,h\in G)
\end{equation}
gives a bijection 
\begin{equation}
\label{e:capxi}
\caP[\xi]\overset{1-1}\longleftrightarrow K_\xi\bs G/B.
\end{equation}

As in
\eqref{e:I} choose a set $\{\xi_i\,|\, i\in I\}$ of representatives of
$\I/G$. Then 

\begin{equation}
\caP\simeq\coprod_{i\in I}(G/K_i\times G/B)/G
\end{equation}
and we see:
\end{subequations}

\begin{proposition}
\label{p:kgb}
There is a natural bijection:
\begin{equation}
\caP(G,\gamma)\overset{1-1}\longleftrightarrow \coprod_{i\in I}\, K_i\bs G/B.
\end{equation}
\end{proposition}

\sec{The One Sided Parameter Space}
\label{s:onesided}
We now turn to the question of formulating an effective algorithm for
computing the space $\caL\subset\caP\times\ch\caP$ \eqref{e:L}.
This mainly comes down to computing the one-sided parameter space
$\caP$, which we do in this section. 
We put $\caP$ and $\ch\caP$ together to define the parameter space
$\caZ$ of representations in Section \ref{s:Z}.

We begin by looking for a normal form for
one-sided L-data.

Fix basic data $(G,\gamma)$ as usual and set $\I=\I(G,\gamma)$.
Recall (cf.~Section \ref{s:kgb}) $\caP=(\I\times\caB)/G$.
Since every Borel subgroup is conjugate to $B$,
every
element of $\I\times\caB$ may be conjugated to one of the form $(\xi,B)$.
Therefore the map
\begin{subequations}
\renewcommand{\theequation}{\theparentequation)(\alph{equation}}  
\begin{equation}
\I\ni \xi\mapsto (\xi,B)\in (\I\times\caB)/G=\caP
\end{equation}
is surjective. Since $B$ is its own normalizer, we see $(\xi,B)$ is
$G$-conjugate to $(\xi',B)$ if and only if $\xi$ is $B$-conjugate to
$\xi'$. So we obtain a bijection
\begin{equation}
\I/B\overset{1-1}\longleftrightarrow \caP
\end{equation}
 from $B$-orbits on $\I$ to $\caP$, 
sending the $B$-orbit of $\xi\in\I$  to the $G$-orbit of the pair $(\xi,B)$.

Now suppose $\xi\in \I$. By 
(\cite[Lemma 6.18]{abv}, \cite{matsuki}), $\xi\in \Norm_{\G^\Gamma}(H_1)$
for some Cartan subgroup $H_1\subset B$.
There exists $b\in B$ such that $bH_1b\inv=H$, so $b\xi b\inv\in
N^\Gamma=\Norm_{\GGamma}(H)$. If $b_1$ is another such element then $b_1=hb$ with $h\in
H$, and $b_1\xi b_1\inv=h(b\xi b\inv)h\inv$. 
Therefore
\begin{equation}
\I/B\simeq (\I\cap N^\Gamma)/H.
\end{equation}
\end{subequations}

This gives our primary combinatorial construction:

\begin{definition}
\label{d:X}  
The  {\it one-sided parameter space} for $(G,\gamma)$ is the set
\begin{equation}
\label{e:X}
\begin{aligned}
\X(G,\gamma)&=(\I\cap N^\Gamma)/H\\
&=\{\xi\in\Norm_{\GGamma\bs G}(H)\,|\, \xi^2\in
Z(G)\}/H.
\end{aligned}
\end{equation}

\end{definition}
This is the set of strong involutions normalizing $H$,
modulo conjugation by $H$.
If $(G,\gamma)$ is understood we write $\X=\X(G,\gamma)$.

From the preceding discussion we have
\begin{equation}
\X=(\I\cap N^\Gamma)/H\simeq \I/B\simeq\caP
\end{equation}
and we conclude:
\begin{proposition}
\label{p:X=S}
There is a canonical bijection
\begin{equation}
\X\overset{1-1}\longleftrightarrow\caP
\end{equation}
taking the $H$-orbit of an element $\xi$ of $\I\cap N^\Gamma$ to the
$G$-orbit of $(\xi,B)$ in $\caP$.
\end{proposition}

Given $x\in\X$, let
\begin{equation}
\label{e:Xx}
\X[x]=\{x'\,|\,x'\text{ is $G$-conjugate to }x\}.
\end{equation}
This is a slight abuse of notation: we say $x,x'\in\X$ are
$G$-conjugate 
if $\xi,\xi'$ are $G$-conjugate, where $\xi,\xi'$ are
pre-images of $x,x'$ in $\I\cap N^\Gamma$, respectively.

By Proposition \ref{p:X=S} and \eqref{e:capxi} we see
\begin{equation}
\label{e:Xxi}
\X[x]\simeq K_\xi\bs G/B
\end{equation}
where $\xi$ is a preimage of $x$ in $\I\cap N^\Gamma$.

Choose a set $\{\xi_i\,|\, i\in I\}$ of
representatives of $\I/G$ as in \eqref{e:I}.
By Proposition \ref{p:kgb} and \eqref{e:Xxi} we obtain:

\begin{corollary}
\label{c:I}
\begin{equation}
\label{e:XcupXi}
\X\simeq\coprod_{i\in I}K_i\bs G/B.
\end{equation}
\end{corollary}

See Examples \ref{ex:sl2_2} and \ref{ex:psl2_1}.

We need to understand the structure of $\X$ in some detail.
We now give more information about it. At the same time we reiterate
some earlier definitions and introduce 
the twisted involutions in the Weyl group.

We fix $(G,\gamma)$ throughout and drop them from the notation.

\begin{subequations}
\renewcommand{\theequation}{\theparentequation)(\alph{equation}}

Let
\begin{equation}
\label{e:x1}
N=\Norm_G(H)\subset N^\Gamma=\Norm_{\GGamma}(H)
\end{equation}
and
\begin{equation}
\label{e:WGamma}
W=N/H\subset W^\Gamma=N^\Gamma/H.
\end{equation}

Recall (Definition \ref{d:stronginvolution}) 
\begin{equation}
\I=\{\xi\in \GGamma\bs G\,|\, \xi^2\in Z(G)\}, 
\end{equation}
and that $G$ acts on $\I$ by conjugation.
Let
\begin{equation}
\begin{aligned}
\label{e:Xt}
\tX&=\I\cap N^\Gamma\\
&=\{\xi\in N^\Gamma\bs N\,|\, \xi^2\in Z(G)\}.
\end{aligned}
\end{equation}
This is the set of strong involutions normalizing $H$.

Let
\begin{equation}
\label{e:X2}
\X=\tX/H
\end{equation}
as in \eqref{e:X}
(the quotient is by the conjugation action).

The group $N$ acts naturally on $\tX$ and $\X$; the  action of $N$ on $\X$
factors to $W$.
This action of $W$ on $\X$ corresponds to the cross action of $W$ on characters
\cite[Definition 8.3.1]{green}. We therefore denote this action
$\times$. 
That is for $w\in W$ and $x\in \X$,
choose $n\in N$ mapping to $w$, $\xi\in\tX$ mapping to $x$ 
and  define 
\begin{equation}
\label{e:cross}
w\times x=\text{ image of }n\xi n\inv\text{ in }\X.
\end{equation}

Every strong involution is conjugate to one in $\tX$, and we see
\begin{equation}
\label{e:txn}
\tX/N\simeq \X/W.
\end{equation}
See Proposition \ref{p:X/W} for an interpretation of this space.

Let 
\begin{equation}
\label{e:I_W}
\I_W=\{\tau\in W^\Gamma\bs W\,|\, \tau^2=1\}.
\end{equation}

Write $\tilde p:\tX\mapsto W^\Gamma$ for the restriction of the
map $N^\Gamma\mapsto W^\Gamma$ to $\tX$.
It is immediate that $\text{Im}(\tilde p)\subset \I_W$, and 
$\tilde p:\tX\mapsto \I_W$ factors to a map
\begin{equation}
\label{e:p}
p:\X\mapsto \I_W.
\end{equation}
\end{subequations}
Recall $\delta$ is the distinguished element of $\X$ (Definition
\ref{d:extended}); use the same notation for its image in $\I_W$.

\begin{lemma}[\cite{richardson_springer}]
\label{l:pX}
The map $p:\X\mapsto \I_W$ is surjective.  
\end{lemma}

We prove this later; see Proposition \ref{p:X_i/W} and the end of
Section \ref{s:cayley}.

\begin{subequations}
\renewcommand{\theequation}{\theparentequation)(\alph{equation}}

For $\xi\in\tX$ the restriction of $\theta_{\xi}$ to $H$ only
depends on the image $x$ of $\xi$ in $\X$. 
Therefore we may define
\begin{equation}
\label{e:thetaxH}
\theta_{x,H}=\theta_{\xi}\text{ restricted to }H.
\end{equation}

By Lemma \ref{l:pX}
$\I_W$ may be thought of as the set of Cartan
involutions  of $H$ for this inner class:
\begin{equation}
\I_W\overset{1-1}\longleftrightarrow \{\theta_{x,H}\,|\, x\in \X\}.
\end{equation}

The map $\xi\mapsto \xi^2\in Z(G)$ is constant on fibers
of the map $\tX\mapsto\X$. For $x\in\X$ we define $x^2\in Z(G)$
accordingly. 
For $z\in Z(G)$ let
\begin{equation}
\label{e:Xz}
\X(z)=\{x\in\X\,|\,x^2=z\}.
\end{equation}
\end{subequations}
Note that $\X(z)$ is empty unless $z\in Z^{\Gamma}$.

We can make these constructions more concrete using the distinguished element
$\delta$ of Definition \ref{d:extended}.
Let $\theta=\int(\delta)$. 
Then
\begin{equation}
\label{e:concrete}
\begin{aligned}
\tX&=\{x\in N\delta\,|\, x^2\in Z(G)\}\\
&=\{g\delta\,|\, g\in N,\, g\theta(g)\in Z(G)\}\\
&\overset{1-1}\longleftrightarrow \{g\in N\,|\, g\theta(g)\in Z(G)\}\\
\X&=\tX/\{g\delta\rightarrow hg\theta(h\inv)\delta\,|\,h\in H\}\\
&\overset{1-1}\longleftrightarrow \{g\in N\,|\, g\theta(g)\in
Z(G)\}/\{g\rightarrow hg\theta(h\inv)\,|\, h\in H\}\\
\I_W&=\{\tau\in W\delta\,|\, \tau^2=1\}\\
&=\{w\delta\,|\, w\in W, w\theta(w)=1\}\\
&\overset{1-1}\longleftrightarrow \{w\in W\,|\, w\theta(w)=1\}
\end{aligned}
\end{equation}
The last equality identifies $\I_W$ with the {\it twisted involutions}
in the Weyl group.
Also note that conjugation in $W^\Gamma$ becomes {\it twisted
  conjugation}:
\begin{equation}
\label{e:twistedconjugation}
y\cdot w=yw\theta(y\inv)\quad (w,y\in W).
\end{equation}

\sec{The Space $\caZ$ and the Main Theorem}
\label{s:Z}

We now describe the parameter space for Harish-Chandra modules of strong
real forms of $G$.
By Theorem \ref{t:ldata} we need to describe the set $\caL$
(Definition \ref{d:L}).
We have done all of the work describing the one-sided parameter space 
$\X(G,\gamma)$
(Definition \ref{d:X}), and now we merely need to put the two sides
together. 

Fix basic data $(G,\gamma)$, and let $(\chG,\ch\gamma)$ be the dual
data (cf.~Section \ref{s:basic}).
Let $\X=\X(G,\gamma)$ and $\ch\X=\X(\chG,\ch\gamma)$.

\begin{definition}
\label{d:Z}
The {\it two-sided parameter space} is the set
\begin{equation}
\label{e:Z}
\caZ(G,\gamma)=\{(x,y)\in\X\times\ch\X\,|\,(\theta_{x,H})^t=-\theta_{y,\ch
  H}\}.
\end{equation}
\end{definition}

We may now state the main result on the parametrization of
admissible representations of real forms of $G$.
This is an immediate consequence of 
Theorem \ref{t:ldata} and the Definitions of $\caL$ and $\caZ$.

Fix a set $\Lambda\subset P_{\text{reg}}$ of representatives of
$P/X^*(H)$. 

\begin{theorem}
\label{t:main}
\hfil\break
\noindent (1) There is a natural bijection
\begin{equation}
\caZ(G,\gamma)\overset{1-1}\longleftrightarrow \Pi(G,\gamma,\Lambda)
\end{equation}
(cf.~Definition \ref{d:pig}).

\medskip\noindent (2) Let $I\simeq \I/G$ be a set of representatives of the
  strong real forms as in \eqref{e:I}. For
  each $i\in I$ let $G_i(\R)$ be the real form of $G$ corresponding to
  the strong involution $\xi_i$ (cf.~Section \ref{s:involutions}).
There is a natural bijection
\begin{equation}
\caZ(G,\gamma)\overset{1-1}\longleftrightarrow \coprod_{i\in I}\Pi(G_i(\R),\Lambda).
\end{equation}

\medskip\noindent (3) Suppose $G$ is adjoint.
Write $G_1(\R),\dots, G_m(\R)$ for the (equivalence classes of) real
forms of $G$ in the given inner class, and choose representatives
$\lambda_1,\dots,\lambda_n\in \Preg$ for $P/R$.
There is a natural bijection:
\begin{equation}
\caZ(G,\gamma)\overset{1-1}\longleftrightarrow \coprod_{i,j}\Pi(G_i(\R),\lambda_j).
\end{equation}
\end{theorem}

For some examples see the end of Section \ref{s:waction}. 
Many more examples are worked out in detail in \cite{adams_snowbird}.

\medskip

We note that $\caZ(G,\gamma)$ can be viewed as a space of orbits as follows.
Let $I$ (resp. $\ch I$) be a set of representatives of $\I/G$ 
(resp. $\ch\I/\chG$), as in \eqref{e:I1}. 
For $i\in I$ (resp. $j\in \ch I$) let $K_i$ (resp. $\ch K_j$) be as in \eqref{e:I2}.
Then 

\begin{equation}
\label{e:Zorbit}
\caZ\subset 
\coprod_{i\in I}K_i\bs G/B\times
\coprod_{j\in \ch I}\ch K_j\bs \chG/\ch B.
\end{equation}

We now give the statement of Vogan Duality \cite{ic4} in this
setting. It is evident that the definition of $\caZ(G,\gamma)$ is
entirely symmetric in $G$ and $\ch G$:
the map $(x,y)\mapsto (y,x)$ is a bijection between
$\caZ(G,\gamma)$ and $\caZ(\chG,\ch\gamma)$.

\begin{corollary}[Vogan Duality]
\label{c:duality}
Fix a set of representatives $\Lambda\subset\Preg$  of
$P(G,H)/X^*(H)$, and
a set $\ch\Lambda\subset \chPreg$ of representatives
of $P(\chG,\ch H)/X^*(\ch H)$.
There is a natural bijection
\begin{equation}
\Pi(G,\gamma,\Lambda)\leftrightarrow \Pi(\chG,\ch\gamma,\ch\Lambda).
\end{equation}
\end{corollary}

This bijection is compatible, in a precise sense, with the 
duality of \cite[Theorem 13.13]{ic4}. 
See \cite{bowdoin} and \cite{abv} for details.
See the Table in example \ref{ex:sl2_2} for the case of
$SL(2)/PSL(2)$, and 
\cite{adams_snowbird} for some more elaborate examples.
\medskip

\sec{Fibers of the map $p:\X\rightarrow \I$}
\label{s:fibers}

In Sections  \ref{s:fibers} through \ref{s:cayley} we  study the space
$\X$ in  more detail,  and relate it  to structure  and representation
theory of real groups. We begin with a study of  the fibers of $\tilde
p$ and $p$. We work in the setting of Section  \ref{s:onesided}.

Fix $\tau\in \I_W$.
Let $\tX_\tau=\wt p\inv(\tau)$ and $\X_\tau=p\inv(\tau)$.
For $z\in Z(G)$ let $\X_\tau(z)=\X_\tau\cap \X(z)=\{x\in \X\,|\,
p(x)=\tau, x^2=z\}$.
Let
\begin{equation}
\begin{aligned}
\label{e:H}
H'_{-\tau}=&\{h\in H\,|\, h\tau(h)\in Z(G)\}\\
H_{-\tau}=&\{h\in H\,|\, h\tau(h)=1\}\\
A_\tau=&\{h\tau(h\inv)\,|\, h\in H\}.
\end{aligned}
\end{equation}
Note that $A_\tau$ is a connected  torus, and is the identity component of $H_{-\tau}$.

\begin{proposition}
\label{p:fibers}
\hfil\break
\begin{enumerate}
\item  $H_{-\tau}'$ acts simply transitively on $\tX_\tau$,
\item $H_{-\tau}'/A_\tau$ acts simply transitively on $\X_\tau$,
\item Fix $z\in Z(G)$. If $\X_\tau(z)$ is non-empty then $H_{-\tau}/A_\tau$ acts simply
  transitively on $\X_\tau(z)$. If $z\not\in Z(G)^\Gamma$ (the
  $\Gamma$-invariants of $Z(G)$) then
  $\X_\tau(z)$ is empty.
\end{enumerate}
In particular  $|\X_\tau(z)|$ is a power of $2$
(or $0$).
If $Z(G)^\Gamma$ is finite 
then $\X$ is a finite set.
\end{proposition}

\begin{proof}
Choose $\xi\in \tX_\tau$.
Then $\tX_\tau=\{h\xi\,|\, h\in H, (h\xi)^2\in Z(G)\}=
\{h\xi\,|\, h\tau(h)\xi^2\in Z(G)\}$.
The first claim follows.

For $h\in H$ we have $h\xi h\inv=h\tau(h\inv)\xi$.
This shows  that the stabilizer in $H_{-\tau}'$,
for the left multiplication action of $H$, 
of the image of $\xi$ in $\X$  
is $\{h\tau(h\inv)\,|\, h\in H\}=A_\tau$. This proves (2), and 
(3) follows immediately from the fact that
$(h\xi)^2=h\tau(h)\xi^2$, and the fact that $\xi^2\in Z(G)^\Gamma$.
The assertion about $|\X_\tau(z)|$ is clear, since 
$H_{-\tau}/A_\tau$ is an elementary abelian two-group.

By (3) $\X$ is the union of the finite sets $\X_\tau(z)$ for $\tau\in
\I_W$ and $z\in Z(G)^\Gamma$, and the final conclusion follows.
\end{proof}

\begin{remarkplain}
\label{r:H2}
Let $T_\tau$ be the identity component of the fixed points of $\tau$
acting on $H$. Then
$T_\tau$ and $A_\tau$ are connected tori,
$H=T_\tau A_\tau$ and $A_\tau\cap T_\tau$ is an elementary abelian two group.
The group in Proposition \ref{p:fibers} (3) is 
\begin{equation}
\label{e:H-tau}
H_{-\tau}/A_\tau\simeq T_\tau(2)/A_\tau\cap T_\tau.
\end{equation}
If we write the real torus corresponding to $\tau$ as 
$(\R^\times)^a\times (S^1)^b\times(\C^\times)^c$ then 
$T_\tau(2)\simeq (\Ztwo)^{b+c}$, $A_\tau\cap T_\tau\simeq (\Ztwo)^c$
and $T_\tau(2)/A_\tau\cap T_\tau\simeq (\Ztwo)^b$.
\end{remarkplain}

\begin{remarkplain}
\label{r:dualH}
It is helpful to note that 
\begin{equation}
H_{-\tau}/A_\tau\simeq[\ch H(\R)/\ch H(\R)^0]^\wedge
\end{equation}
where $\ch H$ is the dual torus to $H$,
with real form $\ch H(\R)$ defined by the Cartan involution $-\tau^\vee$.  
\end{remarkplain}
\sec{Action of $W$ on $\X$}
\label{s:waction}

We now study the action of $W$ on $\X$, which plays an important
role.  We begin with some definitions and terminology.

Fix $\tau\in \I_W$.
Let 
\begin{equation}
\label{e:W}
\begin{aligned}
\Delta_i&=\{\alpha\in \Delta\,|\,\tau(\alpha)=\alpha\}\text{ (the
  imaginary roots)}\\
\Delta_r&=\{\alpha\in \Delta\,|\,\tau(\alpha)=-\alpha\}
\text{ (the real roots)}\\
\Delta_{cx}&=\{\alpha\in \Delta\,|\,\tau(\alpha)\ne\pm\alpha\}
\text{ (the complex roots)}\\
\Delta_i^+&=\Delta_i\cap\Delta^+,\,
\Delta_r^+=\Delta_r\cap\Delta^+\\
W_i&=W(\Delta_i)\\
W_r&=W(\Delta_r).\\
\end{aligned}
\end{equation}
We will write $\Delta_{i,\tau}, W_{i,\tau}$ etc. to indicate the
dependence on $\tau$. We also  refer to the $\tau$-imaginary,
$\tau$-real roots, etc.

Let $\rho_i=\frac12\sum_{\alpha\in\Delta_i^+}\alpha$, and
$\ch\rho_r=\frac12\sum_{\alpha\in\Delta_r^+}\ch\alpha$.
As in 
\cite[Proposition 3.12]{ic4} let
\begin{equation}
\label{e:WC}
\begin{aligned}
\Delta_C&=\{\alpha\in\Delta\,|\, \langle\rho_i,\ch\alpha\rangle
=\langle\alpha,\ch\rho_r\rangle=0\}\subset \Delta_{cx}
\\
W_C&=W(\Delta_C)
\end{aligned}
\end{equation}
This is a complex root system.

Now $\tau$ acts on  $W$, 
 and we let $W^\tau$ be the
fixed points.
By \cite[Proposition 3.12]{ic4} 
\begin{equation}
\label{e:Wtau}
W^\tau=(W_C)^\tau\ltimes (W_i\times W_r).
\end{equation}
Note that $W_i$ and $W_r$ are Weyl groups of the root systems
$\Delta_i$ and $\Delta_r$ respectively; also $(W_C)^\tau$ is
isomorphic to the Weyl group of the root system $(\Delta_C)^\tau$ \cite{ic4}.

Fix $\xi\in\tX_\tau$ and let $\theta=\theta_\xi$, $K=K_\xi$.
For $\alpha\in\Delta_i$ let $X_\alpha$ be an $\alpha$-root vector.
Then  $\theta_{\xi}(X_\alpha)$
only depends on the image $x$ of $\xi$ in $\X$.
We say $\gr_x(\alpha)=0$ if
$\theta_{\xi}(X_\alpha)=X_\alpha$ and 
1 if $\theta_x(X_\alpha) = -X_\alpha$.
This is a $\Ztwo$-grading of $\Delta_i$ in the sense that if
$\alpha,\beta,\alpha+\beta\in\Delta_i$ then
$\gr_x(\alpha+\beta)=\gr_x(\alpha)+\gr_x(\beta)\pmod 2$. 
These are the compact and noncompact imaginary roots.

Let $W(K,H)=\Norm_K(H)/H\cap K$. This is isomorphic to 
$W(\GR,H(\R))$ where $\GR$ is the real form of $G$ corresponding
to $\theta$, and we call it the real Weyl group.
(This contains the Weyl group of the root system of real roots.)
Clearly $W(K,H)\subset W^\tau$. 
Let $M=\Cent_G(A_\tau)$ (cf.~\ref{e:H}).
By 
\cite[Proposition 4.16]{ic4},
\begin{equation}
\label{e:realweyl}
W(K,H)=(W_C)^\tau\ltimes(W(M\cap K,H)\times W_r).
\end{equation}
We have
\begin{equation}
\label{e:wmkh}
W_{i,c}\subset W(M\cap K,H)\simeq W_{i,c}\ltimes \mathcal A(H)\subset W_i
\end{equation}
where $W_{i,c}$ is the Weyl group of the compact imaginary roots
(i.e. $\gr_\xi(\alpha)=0$) 
and
$\mathcal A(H)$ is a certain two-group \cite{ic4}.
This describes $W(K,H)$ in terms of the Weyl groups $(W_C)^\tau$,
$W_r$ and $W_{i,c}$, which are straightforward to compute, and the
two-group $\mathcal A(H)$. For more information on $W(K,H)$ see 
Proposition \ref{p:realweyl}.

Let 
\begin{equation}
\caH=\{(\xi,H_1)\,|\, \xi\in \I, H_1\text{ a }\theta_\xi\text{-stable Cartan subgroup}\}/G.
\end{equation}
With $I,\theta_i$ and $K_i$ as in \eqref{e:I} we may conjugate $\xi$
to some $\xi_i$, and this shows
\begin{equation}
\caH\simeq\coprod_{i\in I}\, \{\text{$\theta_i$-stable Cartan subgroups of $G$}\}/K_i.
\end{equation}
On the other hand every Cartan subgroup is conjugate to $H$, and the normalizer of $H$
is $N$, so 
\begin{equation}
\caH\simeq\I\cap N^\Gamma/N=\tX/N\simeq\X/W
\end{equation}
(cf.~ \eqref{e:txn}).

With notation as in Corollary \ref{c:I} for $i\in I$ we
may assume $\xi_i\in \tX$, and define $x_i\in \X$ and $\X_i=\X[x_i]$
accordingly. We conclude
\begin{proposition}
\label{p:X/W}
For each $i\in I$ we have
\begin{equation}
\X_i/W\leftrightarrow\{\text{$\theta_i$-stable Cartan
  subgroups of $G$}\}/K_i.
\end{equation}
Taking the union over $i\in I$ gives
\begin{equation}
\X/W\leftrightarrow \coprod_i\, \{\text{$\theta_i$-stable Cartan
  subgroups of $G$}\}/K_i.
\end{equation}
\end{proposition}
Recall that by \eqref{e:Xxi} and Proposition \ref{p:kgb} $\X_i\simeq K_i\bs G/B$.

\begin{proposition}[\cite{richardson_springer}]
\label{p:X_i/W}
The map $p:\X_i/W\mapsto \I_W/W$ is injective. If $\theta_i$ is
quasisplit it is a bijection.
\end{proposition}

\begin{remarkplain}
This says that the conjugacy classes of Cartan subgroups of any
real form of $G$ embed in those of the quasisplit form.
See \cite[page 340]{platonov_rapinchuk}.
\end{remarkplain}
\begin{proof}
For injectivity we have to show that 
$\xi,\xi'\in \tX$, $\wt p(\xi)=\wt p(\xi')$ and
$\xi'=g\xi g\inv$ $(g\in G)$ implies
$\xi'=n\xi n\inv$ for some $n\in N$.
The condition $\wt p(\xi)=\wt p(\xi')$ implies $\xi'=h\xi$ for some
$h\in H$,  so $g\xi g\inv=h\xi$, i.e. $g\theta_\xi(g\inv)=h$.
By \cite[Proposition
2.3]{richardson_springer}
there exists $n\in N$ satisfying $h=n\theta_{\xi}(n\inv)$, and then
$\xi'=n\xi n\inv$.

We defer the proof of surjectivity in the quasisplit case to the end
of Section \ref{s:cayley}.
\end{proof}

The real Weyl group (see the discussion
preceding \eqref{e:realweyl}) appears naturally in our setting.
Fix $\xi\in\tX$, let $K=K_\xi$, and let $x$ be the image of $\xi$ in $\X$.

\begin{proposition}
\label{p:realweyl}
$W(K,H)\simeq \Stab_W(x)$.
\end{proposition}

\begin{proof}
We have
\begin{equation}
\begin{aligned}
W(K,H)&=\Norm_K(H)/H\cap K\\
&=\Stab_N(\xi)/\Stab_H(\xi)\\
&=\Stab_N(\xi)H/H.
\end{aligned}
\end{equation}
It is easy to see that $\Stab_N(\xi)H=\Stab_N(x)$, so this equals
$$
\Stab_N(x)/H\simeq \Stab_{N/H}(x)=\Stab_W(x).
$$
\end{proof}

Now fix $\tau\in\I_W$. By Proposition \ref{p:realweyl} and
\eqref{e:realweyl} we see $(W_C)^\tau$ and $W_r$ act trivially on
$\X_\tau$. 
It is worth noting that we can see this directly.

\begin{proposition}
\label{p:trivially}
Both $(W_C)^\tau$ and $W_r$ act trivially on $\X_\tau$.  
\end{proposition}

This proof was communicated to us by David Vogan.

\begin{proof}
Fix $\xi\in\tX_\tau$.
The group $(W_C)^\tau$ is generated by elements $s_\alpha
s_{\tau\alpha}$ where $\alpha\in\Phi_C$. So suppose $\alpha\in\Phi_C$
and let $\sigma_\alpha\in N$ be a preimage of $s_\alpha\in W$.
Let $\sigma_{\tau(\alpha)}=\xi\sigma_\alpha\xi\inv$.
Note that $\alpha+\tau(\alpha)$ is not a root, since it would have to
be imaginary, and (by \eqref{e:WC})
orthogonal to $\rho_i$.
Therefore the root subgroups $G_\alpha$ and $G_{\tau(\alpha)}$
commute. Then $\xi\sigma_\alpha\sigma_{\tau(\alpha)}\xi\inv
=
\sigma_{\tau(\alpha)}\sigma_{\alpha}=\sigma_\alpha\sigma_{\tau(\alpha)}$. 

If $\alpha$ is a $\tau$-real root  this reduces easily
to a computation in $SL(2)$. We omit the details.
\end{proof}

Fix $\tau\in\I_W$ and suppose $x,x'\in\X_\tau$. 
As a consequence of Propositions \ref{p:X_i/W} and \ref{p:realweyl} we
have
\begin{equation}
\label{e:GN}
x'\text{ is $G$-conjugate to }x \Leftrightarrow x'=w\times x\text{ for some
}w\in W_{i,\tau}.
\end{equation}

Another useful result obtained from the action of $W^\tau$ is the
computation of strong real forms.

\begin{proposition}
\label{p:Xstrongrealforms}
Every element $x\in\X$ is $G$-conjugate to an element of
$\X_{\delta}$, and 
there is a canonical bijection between $\X_{\delta}/W_{i,\delta}$ and
the set of strong real forms of $(G,\gamma)$.
\end{proposition}
(See the remark after \eqref{e:Xx} for the notion of $G$-conjugacy.)
This corresponds to the fact that every real form in the given inner
class contains a fundamental (i.e. most compact) Cartan subgroup. Note
that if $G$ is adjoint this gives a bijection between
$\X_\delta/W_{i,\delta}$ and real forms in the given inner class.

We defer the proof until the end of Section \ref{s:cayley}.
See Examples \ref{ex:sl2_2} and \ref{ex:psl2_1}.

\subsec{Recapitulation}
\label{s:recap}
We summarize the main results on the translation between the structure
of $\X$ and some standard objects for $G$.

Fix basic data $(G,\gamma)$ as in Section \ref{s:basic} and define 
$\GGamma$ as in Definition \ref{d:extended}.
Let  $\X=\X(G,\gamma)$  (cf.~\eqref{e:X2}) 
and define
$\I_W$ as in \eqref{e:I_W}.
Fix a 
set $\{\xi_i\,|\, i\in I\}$ of representatives of
the strong real forms, i.e. 
$\I/G$, as in \eqref{e:I}, and 
for $i\in I$ let $\theta_i=\theta_{\xi_i}$ 
and $K_i=K_{\xi_i}$. 
Recall  $\delta$ is the distinguished element of $\GGamma$,
or its image in  $\I_W$.
Also $W_{i,\delta}$ is the Weyl group of the
$\delta$-imaginary roots.
Let $\theta_{qs}$ be a quasisplit involution in this inner class
(cf.~Lemma \ref{l:dandq}) and let $K_{qs}=G^{\theta_{qs}}$.

\begin{proposition}
\label{p:recapitulation}
We have bijections:
\hfil
\begin{enumerate}
\item $\X\overset{1-1}\longleftrightarrow\coprod_{i\in I}K_i\bs G/B$
\,(Corollary  \ref{c:I}),
\item
  $\X_{\delta}/W_{i,\delta}\overset{1-1}\longleftrightarrow\{\text{strong
    involutions}\}/G$=\{strong real forms\}\, (Prop. \ref{p:Xstrongrealforms}),
\item 
$\I_W/W\overset{1-1}\longleftrightarrow \{\text{$\theta_{qs}$-stable
  Cartan subgroups in }G\}/K_{qs}$
\,(Prop. \ref{p:X_i/W}),
\item 
$\X/W\overset{1-1}\longleftrightarrow \coprod_i\{\text{$\theta_i$-stable Cartan subgroups in
    $G$}\}/K_i$\,(Prop. \ref{p:X/W}),
\item $\X_{\tau}(z)\simeq [\ch H(\R)/\ch H(\R)^0]^\wedge$ \text{or is
    empty} \,
  $(\tau\in\I_W,z\in Z$,$\ch H(\R)$ as in Remark \ref{r:dualH}),
\item $\Stab_W(x)\simeq W(K_\xi,H)$\, $(\xi\in \tX,\text{ with image
    }x\in \X)$\quad(Prop. \ref{p:realweyl}).
\end{enumerate}
\end{proposition}

\begin{exampleplain}
\label{ex:sl2_2}
We illustrate each part of the Proposition in the case of
$SL(2,\C)$.
In this case $\Out(G)=1$ so
there is only one inner class of involutions, and we drop $\delta$
from the notation.  See examples \ref{ex:sl2C} and \ref{ex:sl2_1}.

Write $H=\{\diag(z,\frac1z)\,|\, z\in \C^\times\}$,
$W=\{1,s\}$, and let $t=\diag(i,-i)$.
Let $n$ be any element of $\Norm_G(H)$ mapping to $s\in W$.
Then $\Norm_G(H)=H\cup Hn$, and
\begin{equation}
\begin{aligned}
\tX&=\{\xi\in H\cup Hn\,|\, \xi^2=\pm \Id\}\\
&=\{\pm \Id,\pm t\}\cup Hn\\
\X&=\{\pm \Id,\pm t\}/H\cup Hn/H\\
\end{aligned}
\end{equation}
Note that $H$ acts trivially on $\{\pm \Id,\pm t\}$, 
and $hnh\inv=h^2n$ for all $h\in H$, so $Hn/H$ is a singleton.
Therefore 
\begin{equation}
\X=\{\pm \Id,\pm t,n\}.
\end{equation}
Strictly speaking these are elements of $\tX$ representing $\X$.

Since $\delta=1$ we have $\X_{\delta}=\{\pm \Id,\pm t\}$, and
$W_{i,\delta}=W$. 
Part (2) of the Proposition says we can take $I$ to be a set of
representatives of $\X_\delta/W=\{\pm \Id,t\}$.
Recall (\ref{ex:sl2_1}) we think of these as ``strong real forms'' 
$SU(2,0),SU(0,2)$ and $SU(1,1)\simeq SL(2,\R)$, respectively.
Then
\begin{equation}
\label{e:XI}
\X[\Id]=\{\Id\},\quad \X[-\Id]=\{-\Id\},\quad \X[t]=\{t,-t,n\}.
\end{equation}
Now $G/B$ is isomorphic to the complex projective plane
$\C\cup\{\infty\}$.
We have $K_{\pm \Id}=G$ and $K_{\pm \Id}\bs G/B$ is a point.
We label these orbits $\O_{2,0}$ and $\O_{0,2}$, respectively. 
On the other hand $K_t\simeq\C^\times$, which acts on $G/B$ by
$z:u\mapsto z^2u$. Therefore there are three orbits of this action:
$\O_0=\{0\}, \O_\infty=\{\infty\}$ and $\O_*=\C^\times$.

So the bijection of (1) is

\bigskip
\begin{equation}
\label{e:orbitssl2}
{\renewcommand{\arraystretch}{1.5}
\begin{tabular}{|c|c|c|c|c|c|}\hline
$x\in\X$& \Id& -\Id & t & -t & n\\\hline
$K_x$& G&G&$\C^\times$&$\C^\times$&$\C^\times$\\\hline
Orbit& $\O_{2,0}$&$\O_{0,2}$&$\O_{0}$&$\O_{\infty}$&$\O_*$\\\hline
\end{tabular}
}
\end{equation}
\bigskip

In this case $\I_W=W=\{1,s\}$. The quasisplit group is $SL(2,\R)$,
which has two conjugacy classes of Cartan subgroups. The compact
Cartan subgroup $T\simeq S^1$ corresponds to $1\in W$, and the split
Cartan subgroup $A\simeq\R^\times$
corresponds to $s\in W$. This is (3) in this case.

Now $\X/W$ has four elements $I,-I,t$ and $n$, corresponding to the
compact Cartan subgroups of $SU(2,0), SU(0,2), SU(1,1)$, and the split
Cartan subgroup of $SU(1,1)$, respectively. This is the content of
(4).

For (5) we have $\X_1(\Id)=\{\pm \Id\}$ and $\X_1(-\Id)=\{\pm t\}$. In this
case $H(\R)\simeq S^1$, and $\ch H(\R)\simeq \R^\times$, so $\ch
H(\R)/\ch H(\R)^0\simeq\Ztwo$.
On the other hand $\X_s(\Id)=\emptyset$ and $\X_s(-\Id)=\{n\}$. In this
case $H(\R)=\R^\times$ and $\ch H(\R)=S^1$ is connected.

Finally consider (6). We have $\Stab_W(\pm \Id)=W$, and $\Stab_W(\pm
t)=1$. This corresponds to the fact that $W(SU(2),S^1))=W$, 
and $W(SL(2,\R),S^1)=1$. On the other hand $\Stab_W(n)=W$,
i.e. $W(SL(2,\R),\R^\times)=W$. 

Most of the conclusions of the  Proposition 
seen in Figure 1.
Projection $p:\X\mapsto \I_W$ is written vertically.

\begin{figure}[ht]
\begin{equation*}
\xymatrix @=0pt @*=<1.5cm,1cm>{
\save []*++!CR{SU(2,0)} \restore \ar@{-}[]!RU+0;[5,0]!RD+0 
\ar@{-}[]!LD+0;[0,2]!RD+0
&\Id&\ar@{-}[]!RU+0;[5,0]!RD+0 
&
\save [].[1,0]!C*--\frm{\}},!R(1.1)*+++!L!/d .5pt/{\longrightarrow z=\Id}\restore \\
\ar@{-}[]!LD+0;[0,3]!RD+0
\save []*++!CR{SU(0,2)} \restore &-\Id&&\\
\save [].[1,0]!C*---\frm{\{}*++++++++++!CR{SU(1,1)} \restore &t&&
\save [].[1,0]!C*--\frm{\}},!R(1.1)*+++!L!/d .5pt/{\longrightarrow z=-\Id}\restore \\
\ar@{=}[]!LD+0;[0,3]!RD+0
&-t&w&&&&\\
\I_W&1&s_\alpha&&&\\
\text{Cartan}&T&A&&&\\
}
\end{equation*}
\caption{$\X$ for $SL(2)$}
\end{figure}
\end{exampleplain}

\newpage

\begin{exampleplain}
\label{ex:psl2_1}
We reconsider the previous example with $PSL(2,\C)$ in place of
$SL(2,\C)$. Again $\gamma=1$  and we drop it from the notation.

Recall $PSL(2,\C)\simeq SO(3,\C)$, and it is easier to work with the latter
realization, with respect to the form 
$\begin{pmatrix}
0&1&0\\
1&0&0\\
0&0&1\\  
\end{pmatrix}$.
We  take $H=\{\diag(z,\frac1z,1)\,|\,z\in\C^\times\}$, write
$W=\{1,s\}$  and let $n$ be a representative in $\Norm_G(H)$ of $s$.
Let $t=\diag(-1,-1,1)$.

As 
in the previous example we have:
\begin{equation}
\begin{aligned}
\X&=\{\xi\in H\cup Hn\,|\, \xi^2=\Id\}/H\\
&=\{\Id,t\}/H\cup (Hn)/H\\
&=\{\Id,t,n\}.
\end{aligned}
\end{equation}
In this case  we can take $I=\{\Id,t\}$, and
\begin{equation}
\X[\Id]=\{\Id\}, \X[t]=\{t,n\}.
\end{equation}
Then $K_{\Id}=G$ and $K_{\Id}\bs G/B$ is a point. On the other hand $K_t\simeq
O(2,\C)$, which has two orbits on the projective plane: $\{0,\infty\}$
and $\C^\times$. This is  (1) of the Proposition in this case.

In this  case (2) says that 
$\X_1/W=\{\Id,t\}$, corresponding to the two real forms of
$G$.

The analogue of 
\eqref{e:orbitssl2} is
$$
\label{e:orbitspgl2}
{\renewcommand{\arraystretch}{1.5}
\begin{tabular}{|c|c|c|c|}\hline
$x\in\X$& \Id& t & n\\\hline
$K_x$& $G$&$\C^\times$&$\C^\times$\\\hline
Orbit& $\O_{2,0}$&$\O_{0}$&$\O_*$\\\hline
\end{tabular}
}
$$
Statement (3) is the same as for $SL(2,\C)$: $\I_W/W$ has two
elements, corresponding to the two conjugacy classes of  Cartan
subgroups of $SO(2,1)$.

For (4), $\X/W=\X=\{\Id,t,n\}$; corresponding to the compact Cartan
subgroups of $SO(3), SO(2,1)$, and the split Cartan subgroup of
$SO(2,1)$, respectively.

Next, $\X_1(\Id)=\{\Id,t\}$ and $\X_s(I)=\{n\}$, corresponding to $\ch
H(\R)=\R^\times$ and $S^1$ as in the previous example. This gives (5).

Finally note that $\Stab_W(1)=\Stab(n)=W$ as in the previous example.
However
$\Stab_W(t)=W$; i.e. $W(SO(2,1),S^1)=W$.
Comparing this case with the fact that $\Stab_W(t)=1$ in the case of $SL(2,\C)$
 illustrates
how $\Stab_W(\xi)$ depends  in  a subtle way on   isogenies.
See \eqref{e:wmkh}.
\end{exampleplain}

\begin{exampleplain}
\label{ex:sl2psl2reps}
Having described $\X$ for $SL(2,\C)$ and $PSL(2,\C)$ we can now
describe the representation theory of real forms of these groups in
terms of the space $\caZ$. See Theorem \ref{t:main}, and Examples
\ref{ex:sl2_2} and \ref{ex:psl2_1}.
Note that the representations are parametrized by pairs of orbits as
in \eqref{e:Zorbit}.

Write $\C$ for the trivial representation.

For $SU(1,1)\simeq SL(2,\R)$, at infinitesimal character $\rho$, write  $DS_\pm$ for the discrete
series representations and 
$PS_{odd}$ for the irreducible (non-spherical) principal series
representation. 

Consider $SO(2,1)$.
Let $sgn$ be the sign representation of $SO(2,1)$, and $DS$ be the
unique discrete series representation with infinitesimal character
$\rho$.
At infinitesimal character $2\rho$ $SO(2,1)$ has two irreducible 
principal series 
representations  denoted
$PS_{\pm}$. 

\newpage

\centerline{\bf\large Table of representations of $SL(2)$ and
  $PGL(2)$}
\medskip
\centerline{
{\renewcommand{\arraystretch}{2}
\begin{tabular}{| l| l| l| l| l| l| l|| l| l| l| l| l|l|l|}
\hline
Orbit & x & $x^2$  & $\theta_x$ & $G_x(\R)$ & $\lambda$ & rep&
Orbit & y &  $y^2$ & $\theta_y$ & $\ch G_y(\R)$ & $\lambda$ & rep\\
\hline
$\O_{2,0}$ & \Id & \Id &1 & $SU(2,0)$ & $\rho$ & $\C$ &
$\O'_{*}$ & n & \Id &-1 & $SO(2,1)$ & $2\rho$ &$PS_{+}$\\ 
\hline
$\O_{0,2}$ & -\Id & \Id&1 & $SU(0,2)$ & $\rho$ & $\C$ &
$\O'_{*}$ & n & \Id&-1 & $SO(2,1)$ & $2\rho$ &$PS_{-}$\\
\hline
$\O_{0}$ & t & -\Id&1 & $SU(1,1)$ & $\rho$ & DS$_+$ &
$\O'_{*}$ & n & \Id&-1 & $SO(2,1)$ & $\rho$ & $\C$ \\ 
\hline
$\O_{\infty}$ & -t & -\Id&1 & $SU(1,1)$ & $\rho$ & DS$_-$ &
$\O'_{*}$ & n & \Id&-1 & $SO(2,1)$ & $\rho$ & sgn \\ 
\hline
$\O_{*}$ & n & -\Id&1 & $SU(1,1)$ & $\rho$ & $\C$ &
$\O'_{+}$ & t & \Id&-1 & $SO(2,1)$ & $\rho$ & DS \\ 
\hline
$\O_{*}$ & n & \Id&1 & $SU(1,1)$ & $\rho$ & $PS_{odd}$ &
$\O'_{3,0}$ & \Id & \Id&1 & $SO(3)$ & $\rho$ & $\C$ \\
\hline
\end{tabular}
}
}
\bigskip\bigskip
See \cite{adams_snowbird} for more detail on this example.
\end{exampleplain}

\sec{The reduced parameter space}
\label{s:reduced}
If $G$ is adjoint we saw in Section \ref{s:Z} that the parameter space
$\X=\X(G,\gamma)$ is perfectly suited to parametrizing representations
of real forms of $G$.  If $G$ is not adjoint
then strong involutions play an essential role, and the difference
between involutions and strong involutions is unavoidable.
Nevertheless in some respects the space $\X$ is larger than necessary,
and a satisfactory theory is obtained with a smaller set, the {\it
  reduced one-sided parameter space}.  While $\X$ may be infinite,
this is always a finite set.

The most economical possibility would be to keep a single orbit of
strong involutions over each orbit of real forms.  
However there is no canonical
way to make this choice.  The reduced parameter space 
provides a more canonical way to reduce to a small finite number of
choices related to the center.

Let $Z=Z(G)$, and recall (cf. \ref{p:fibers}(3)) $Z^\Gamma$ is the
$\Gamma$-invariants in $Z$.
There is a natural action of $Z$ on $\X$ by left multiplication.
This preserves the fibers $\X_\tau$, and commutes with the conjugation
action of $G$.
For $x\in \X$ recall \eqref{e:Xx} $\X[x]$ is the set of elements of $\X$ conjugate to
$x$. 
If  $z\in Z$ multiplication by $z$ is a
bijection:
\begin{equation}
\X[x]\overset{1-1}\longleftrightarrow \X[zx]
\end{equation}
(cf.~\eqref{e:Xx}).
In other words the  orbit pictures for $x$ and $zx$ are
identical. 
Suppose $\xi\in \tX$ lies over $x$.
Harish-Chandra modules for  $\xi$ are $(\g,K_\xi)$
modules; 
since $K_\xi=K_{z\xi}$, Harish-Chandra modules for $\xi$ 
are exactly the same as Harish-Chandra modules for $z\xi$, with the
same notion of equivalence.

For example suppose $G=SL(2)$ and take  $\xi=I$ and $z=-I$. Then $\xi=I$ and
$z\xi=-I$ both correspond to the compact group $SU(2)$. See 
Example \ref{ex:sl2_1}.

Write $\theta$ for the action
of the non-trivial element of $\Gamma$ on $Z$.
For $z\in Z$ recall $\X(z)$ is the set of elements of $\X$ satisfying
$x^2=z$ \eqref{e:Xz}, and $\X(z)$ is empty unless $z\in Z^\Gamma$. 
If $x\in \X(z')$ and $z\in Z$ then
\begin{equation}
zx\in \X(z'z\theta(z)).
\end{equation}
It is easy to see that
\begin{equation}
Z^\Gamma/\{z\theta(z)\,|\,z\in Z\}\simeq H^2(\Gamma,Z)
\end{equation}
is a finite set. This comes down to the fact that
if $Z$ is a torus then $Z^\Gamma/\{z\theta(z)\}\simeq(\Ztwo)^n$ where
$n$ is the number of $\R^\times$ factors in the corresponding real
torus (cf.~Remark \ref{r:H2}).

\begin{definition}
\label{d:reduced}
Choose a set of representatives $Z_0\subset Z^\Gamma$ for
$Z^\Gamma/\{z\theta(z)\}\simeq H^2(\Gamma,Z)$.
The {\it reduced parameter space} is 
\begin{equation}
\label{e:reduced}
\X^r(G,\gamma)=\coprod_{z\in Z_0}\X(z).
\end{equation}
\end{definition}

\begin{exampleplain}
Let $G=SL(n,\C)$, and let $\gamma=1$. Suppose $p+q=n$ and 
$\alpha^n=(-1)^q$. Let
$$
\xi_\alpha=\diag(\overbrace{\alpha,\dots,\alpha}^p,\overbrace{-\alpha,\dots,-\alpha}^q)
$$
These are representatives of the equivalence classes of strong
involutions in this inner class.  For fixed $p\ne q$ there are $n$
strong real forms, all mapping to the  real
form $SU(p,q)$.  In other words we obtain $n$ identical orbit
pictures. If $p=q$ a similar statement holds, except that
$\xi_\alpha$ is conjugate to $\xi_{-\alpha}$.

If  $n$ is odd then $Z^\Gamma/\{z\theta(z)\}=Z/Z^2$ is trivial,
so we take $Z_0=\{I\}$, and the equivalence classes of strong
involutions in $\X_0(G,\gamma)$ are represented by
\begin{equation}
(\overbrace{1,\dots,1}^p,\overbrace{-1,\dots,-1}^q)
\quad (q\text{ even}).
\end{equation}
For each $p,q$ there is a unique strong involution mapping to the real
form $SU(p,q)$, instead of $n$ as we had earlier.

If $n$ is even then $Z^\Gamma/\{z\theta(z)\}$ has order $2$.
We can take
$Z_0=\{I,\zeta I\}$ where $\zeta$ is a primitive $n^{th}$ root of
$1$. Let $\tau$ be a primitive $2n^{th}$ root of $1$. 
Then the strong real forms in $\X_0(G,\gamma)$
are 
\begin{equation}
\begin{aligned}
\pm\diag(\overbrace{1,\dots,1}^p,\overbrace{-1,\dots,-1}^q)&\quad (q\text{ even},p\ne q)\\
\diag(\overbrace{1,\dots,1}^p,\overbrace{-1,\dots,-1}^p)&\quad\\
\pm\diag(\overbrace{\tau,\dots,\tau}^p,\overbrace{-\tau,\dots,-\tau}^q)&\quad (q\text{ odd}).
\end{aligned}
\end{equation}

In this case there are two strong real forms mapping to the real form
$SU(p,q)$ if $p\ne q$, and $1$ if $p=q$.
\end{exampleplain}
\medskip

The calculations needed to understand representation theory 
(see Section \ref{s:Z}) take place entirely in a fixed set
$\X(z)$. The sets $\X(z')$ and $\X(z'z\theta(z))$ are canonically
identified, so it is safe to think of $\X(z)$ as being defined for
$z\in Z_0$. The {\tt atlas} software takes this approach.

\sec{Cayley Transforms and the Cross Action}
\label{s:cayley}
We continue to work with the one-sided
parameter space $\X=\X(G,\gamma)$. We begin with some formal constructions. 

Fix $x\in\X$ and let $\tau=p(x)\in \I_W$.
Recall (Section \ref{s:waction}) $\tau$ defines the real,
imaginary and complex roots, and $x$ defines a grading
$\gr_{x}$
of the imaginary roots.
Suppose $\alpha$ is an imaginary noncompact root,
i.e. $\tau(\alpha)=\alpha$ and $\gr_x(\alpha)=1$.

Let $G_\alpha$ be the derived group of $\Cent_G(\ker(\alpha))$, and
$H_\alpha\subset G_\alpha$ the one-parameter subgroup corresponding to
$\alpha$. Then $G_\alpha$ is isomorphic to $SL(2)$ or $PSL(2)$ and
$H_\alpha$ is a Cartan subgroup of $G_\alpha$. 
Choose $\sigma_\alpha\in \Norm_{G_\alpha}(H_\alpha)\bs H_\alpha$; then
$\sigma_\alpha(\alpha)=-\alpha$. Let $m_\alpha=\sigma_\alpha^2=\ch\alpha(-1)$.

\begin{definition}
\label{d:cayley}
Suppose $x\in\X$ and $\alpha$ is a noncompact imaginary root.
Choose a pre-image
$\xi$ of $x$ in $\tX$, and
define $c^\alpha(x)$ to be the image of $\sigma_\alpha\xi$ in $\X$.
\end{definition}

\begin{lemma}
\label{l:cayley1}
\hfil
\begin{enumerate}
\item $c^\alpha(x)$ is well defined, independent of the choices of
  $\sigma_\alpha$ and $\xi$.
\item $c^\alpha(x)$ is $G$-conjugate to $x$, and $c^\alpha(x)^2=x^2$.
\item $p(c^\alpha(x))=s_\alpha p(x)\in \I_W$.
\end{enumerate}
\end{lemma}
(Here and in Lemma \ref{l:cayley2} keep in mind  the remark after
\eqref{e:Xx} regarding the notion of $G$-conjugacy.) 

\begin{proof}
Choose $\xi$, and let $t=\ch\alpha(i)\in H_\alpha$. 
Suppose $h\in H_\alpha$. We have a few elementary identities,
essentially in $SL(2)$:

\begin{equation}
\label{e:sl2ids}
\begin{aligned}
\sigma_\alpha h\sigma_\alpha\inv&=h\inv,\quad
h\sigma_\alpha h\inv =h^2\sigma_\alpha\\
\xi h\xi\inv&=h\\
tgt\inv&=\xi g\xi\inv \quad(g\in G_\alpha)\\
\xi\sigma_\alpha \xi\inv&=\sigma_\alpha\inv.
\end{aligned}
\end{equation}
The first two lines follow from 
$\sigma_\alpha(\ch\alpha)=-\ch\alpha$
and $\theta_{\xi}(\ch\alpha)=\ch\alpha$.
For the third, $\int(t)$ and $\theta_\xi$ 
agree on $G_\alpha$, since they agree on
$H_\alpha$ and the $\pm\alpha$ root spaces. The final assertion follows
from the third and a calculation in $SL(2)$.

Now $\sigma_\alpha\xi$ clearly normalizes $H$, and
$$
(\sigma_\alpha \xi)^2=\sigma_\alpha (\xi\sigma_\alpha\xi\inv)\xi^2=\xi^2\in Z(G),
$$
so $\sigma_\alpha\xi\in \tX$.

Given a choice of $\sigma_\alpha$ any other choice is of the form
$h^2\sigma_\alpha$ for some
$h\in H_\alpha$, and 
\begin{equation}
(h^2\sigma_\alpha)\xi=(h\sigma_\alpha h\inv)\xi=h\sigma_\alpha(h\inv
\xi h)h\inv=h(\sigma_\alpha \xi)h\inv.
\end{equation}
Therefore the image of $\sigma_\alpha\xi$ in $\X$ is independent
of the choice of $\sigma_\alpha$.

We need to show that 
$\sigma_\alpha\xi$ and
$\sigma_\alpha h\xi h\inv$ have the same image in $\X$ for all $h\in H$.
Write $H=H_\alpha(\ker(\alpha))$. If $h\in H_\alpha$ then 
$h\xi h\inv =\xi$ so the assertion is  obvious. If $h\in\ker(\alpha)$ then
$\sigma_\alpha h=h\sigma_\alpha$, and 
$\sigma_\alpha(h\xi h\inv)=h(\sigma_\alpha \xi)h\inv$.

For the second assertion, we  actually show $c^\alpha(x)$ is conjugate
to $x$ by an element of $G_\alpha$.
By a calculation in $SL(2)$ it is easy to see $g(\sigma_\alpha t)g\inv =t$ for some $g\in G_\alpha$.
Therefore
\begin{equation}
g(\sigma_\alpha \xi)g\inv=
g(\sigma_\alpha tt\inv\xi)g\inv
=
g(\sigma_\alpha t)g\inv g(t\inv \xi)g\inv =
tt\inv \xi=\xi.
\end{equation}
The fact that $c^\alpha(x)^2=x^2$ follows immediately, and the final assertion is obvious.
\end{proof}

We now define inverse Cayley transforms.
Suppose $\xi\in\tX$, and let $\tau=\tilde p(\xi)$.
Suppose $\alpha$ is a $\tau$-real root.
Define $G_\alpha$ and $H_\alpha$ as before. Let $m_\alpha=\ch\alpha(-1)$.

\begin{lemma}
\label{l:sigmaalpha}
There exists $\sigma_\alpha\in \Norm_{G_\alpha}(H_\alpha)\bs H_\alpha$ so that
$\sigma_\alpha\xi=g\xi g\inv$ for some $g\in G_\alpha$.
The only other element satisfying these conditions is
$m_\alpha\sigma_\alpha$. 
\end{lemma}

\begin{proof}
This is similar to the previous case.
The involution $\theta_\xi$ restricted to $G_\alpha$ is inner for
$G_\alpha$, and acts by $h\mapsto h\inv$ for $h\in H_\alpha$. 
Therefore we may choose $y\in \Norm_{G_\alpha}(H_\alpha)\bs H_\alpha$ so that
$ygy\inv=\xi g\xi\inv$ for all $g\in G_\alpha$.
By a calculation in $SL(2)$ we may choose $\sigma_\alpha$ so that
$g(\sigma_\alpha y)g\inv =y$ for some $g\in G_\alpha$. 
Then
\begin{equation}
g(\sigma_\alpha\xi)g\inv 
=
g(\sigma_\alpha yy\inv \xi)g\inv
=
g(\sigma_\alpha y)g\inv g(y\inv \xi)g\inv 
=
yy\inv \xi=\xi.
\end{equation}
We have $\sigma_\alpha y\in H_\alpha$, and
$\alpha(\sigma_\alpha y)=-1$. Therefore any two such choices
differ by $m_\alpha$.
\end{proof}

\begin{definition}
\label{d:cayley2}
Suppose $\xi\in\tX$ and $\alpha$ is a real root with respect to
$\theta_{\xi}$. Let $c_\alpha(\xi)=
\{\sigma_\alpha\xi,m_\alpha\sigma_\alpha\xi\}$.

If $x\in \X$ choose $\xi\in \tX$ mapping to $x$, and define
$c_\alpha(x)$ to be the image of $c_\alpha(\xi)$ in $\X$.
This is a set with one or two elements.
\end{definition}

The analogue of Lemma \ref{l:cayley1} is immediate:

\begin{lemma}
\label{l:cayley2}
Suppose $x\in \X$ and $\alpha$ is a real root with respect to
$\theta_x$.
\hfil
\begin{enumerate}
\item $c_\alpha(x)$ is well defined, independent of the choice of
  $\xi$.
\item If $y\in c_\alpha(x)$ then 
$y$ is $G$-conjugate to $x$, and $y^2=x^2$.
\item $p(c_\alpha(x))=s_\alpha p(x)\in \I_W$.
\end{enumerate}
\end{lemma}

We deduce some simple formal properties of Cayley transforms.
Fix $\tau\in\I_W$.
If $\alpha$ is $\tau$-imaginary let
\begin{equation}
\begin{aligned}
\X_\tau(\alpha)&=\{x\in \X_\tau\,|\, \alpha\text{ is noncompact with
respect to }\theta_x\}\\
&=\{x\in \X_\tau\,|\, \gr_x(\alpha)=1\}\\
&=\{x\in \X_\tau\,|\, c^\alpha(x)\text{ is defined}\}.\\
\end{aligned}
\end{equation}

\begin{lemma}
\label{l:cayley3}

\hfil
\begin{enumerate}
\item If $\tau(\alpha)=-\alpha$ then for all $x\in \X_\tau$, 
$c^\alpha(c_\alpha(x))=x$,
\item If $\tau(\alpha)=\alpha$ and $x\in \X_\tau(\alpha)$ then 
  $c_\alpha(c^\alpha(x))=\{x,m_\alpha x\}$.
\item The map $c^\alpha: \X_\tau(\alpha)\mapsto \X_{s_\alpha\tau}$
  is surjective, and at most two-to-one
\item Suppose $\alpha$ is imaginary.
The following conditions are equivalent:
\begin{enumerate}
\item $c^\alpha:\X_\tau(\alpha)\rightarrow \X_{s_\alpha\tau}$ is a bijection;
\item $c_\alpha:\X_{s_\alpha\tau}\rightarrow \X_{\tau}(\alpha)$ is a
  bijection;
\item $c_\alpha(x)$ is single valued for all $x\in \X_{s_\alpha\tau}$;
\item $m_\alpha\in A_\tau$ (cf.~\eqref{e:H});
\item $s_\alpha\in W(K_\xi,H)$ for all $\xi\in\tX$ with image in $\X_\tau(\alpha)$,
\item $x=m_\alpha x$ for all $x\in \X_\tau(\alpha)$.
\end{enumerate}
If these conditions fail then $c^\alpha:\X_\tau(\alpha)\rightarrow
\X_{s_\alpha\tau}$ is two to one, and $c_\alpha(x)$ is double valued
for all $x\in \X_{s_\alpha\tau}$.
\item Suppose $\alpha$ is imaginary with respect to $\tau$.
If there exists $h\in H_{-\tau}'$ 
(cf.~\eqref{e:H}) such that $\alpha(h)=-1$ then
$\X_\tau$ is the disjoint union of $\X_\tau(\alpha)$ and $h\X_\tau(\alpha)$.
Otherwise $\X_\tau(\alpha)=\X_\tau$.
\end{enumerate}
\end{lemma}

We leave the straightforward proof to the reader.

\begin{remarkplain}
\label{r:computeX}
Using this Lemma it is straightforward to compute the space $\X$,
starting with $\X_{\delta}$, and computing the fibers $\X_\tau$
inductively. This shows that it is in fact easier to
understand the entire space $\X$ rather than the individual pieces
$K_i\bs G/B\subset\X$ (cf.~Corollary \ref{c:I}).
We describe the latter in more detail in the next section.
\end{remarkplain}

It is important to understand the effect of Cayley transforms on the
grading of the imaginary roots. 
This is due to Schmid \cite{schmid_ds}; also see 
\cite[Definition 5.2 and Lemma 10.9]{ic4}.

\begin{lemma}
\label{l:grading}
Suppose $\tau\in \I_W$ and $x\in \X_\tau(\alpha)$.
Then the $s_\alpha\tau$-imaginary roots are the $\tau$-imaginary roots orthogonal
to $\alpha$, and for such a root $\beta$,
\begin{equation}
\gr_{c^\alpha x}(\beta)=
\begin{cases}
  \gr_x(\beta)&\text{if $\alpha+\beta$ is not a root}\\  
\gr_x(\beta)+1&\text{if $\alpha+\beta$ is a root}.\\  
\end{cases}
\end{equation}
\end{lemma}

\begin{remarkplain}
\label{r:grading}
Choose a pre-image $\xi$ of $x$ in $\tX$.
By 
\eqref{e:sl2ids} we have
\begin{equation}
\label{e:xisigmaalpha}
\xi\sigma_\alpha\xi\inv=
\begin{cases}
\sigma_\alpha&\alpha\text{ compact}\\
m_\alpha\sigma_\alpha&\alpha\text{ noncompact.}
\end{cases}
\end{equation}
If (the derived group of) $G$ is simply
connected then $m_\alpha\ne1$, so
$\gr_x(\alpha)=0$ if $\xi\sigma_\alpha\xi\inv=\sigma_\alpha$, and $1$ otherwise.
A calculation in rank $2$ shows that if $\alpha$ and $\beta$ are
orthogonal then
$\sigma_\alpha\sigma_\beta\sigma_\alpha\inv=\sigma_\beta^{\pm1}$
depending on whether $\alpha+\beta$ is a root or not (see the next
section). 
The Lemma follows readily from this.
\end{remarkplain}

Recall \eqref{e:cross}  $W$ acts on $\X$, and we refer to this as
the cross action.

\begin{lemma}
\label{l:cross}
Suppose $\alpha$ is imaginary with respect to $x$. Then $w\alpha$
is imaginary with respect to $w\times x$ and 
\begin{subequations}
\renewcommand{\theequation}{\theparentequation)(\alph{equation}}  
\begin{equation}
\label{e:crossgrading}
\gr_{w\times x}(w\alpha)=\gr_x(\alpha).
\end{equation}
Suppose $gr_x(\alpha)=1$. Then
\noindent
\begin{equation}
c^{w\alpha}(w\times x)=w\times c^\alpha(x).
\end{equation}
\end{subequations}
\end{lemma}

We leave the elementary proof, and the statement of the corresponding
facts for real roots, to the reader. See \cite{ic4}, Lemmas 4.15 and
7.11. 

With Cayley transforms in hand we can complete the proof of
Proposition \ref{p:X_i/W}.

\begin{proof}[Proof of Propositions \ref{p:X_i/W} and \ref{p:Xstrongrealforms}]
Fix $\tau\in \I_W$. Assume there is a $\tau$-imaginary root $\alpha$.
By Lemma \ref{l:cayley3}(5) there exists $x\in \X_\tau(\alpha)$, so
$c^\alpha(x)\in\X_{s_\alpha\tau}$ is defined. 
Now suppose $\beta$ is an imaginary root with respect to $s_\alpha\tau$. 
By the same argument we may choose $x'\in \X_{s_\alpha\tau}$ so that
$x''=c_\beta(x')$ is defined. Replacing $x\in\X_\tau$ with
$c_\alpha(x')\in\X_\tau$ we now have $\X_\tau\ni x\rightarrow
c^\beta c^\alpha x\in\X_{s_\beta s_\alpha\tau}$.
By Lemma \ref{l:cayley1}(2) $c^\beta c^\alpha(x)$ is $G$-conjugate to $x$.

Continue in this way until we obtain $x\in \X_\tau,x'\in\X_{\tau'}$,
where $x'$ is $G$-conjugate to $x$, 
and there are no imaginary roots with respect to $\tau'$.
(This corresponds to the most split Cartan subgroup of the quasisplit
form of $G$.)
By \cite[Proposition 6.24]{av1} $\theta_{\xi'}$ is quasisplit, for any
$\xi'\in \tX$ lying over $x'$.

This completes the proof of  Proposition \ref{p:X_i/W}. The proof of Proposition
\ref{p:Xstrongrealforms} is similar, using inverse Cayley
transforms,
Proposition \ref{p:trivially} and \eqref{e:GN}.
We omit the details.
\end{proof}

\begin{exampleplain}
\label{ex:sp4}
We conclude this section with some details in the case of 
$G=Sp(4)$ (of type $C_2$). We give a picture of the space $\X$ and
describe the action of  Cayley transforms on $\X$.

There are four $G$-conjugacy classes of strong involutions, which we
think of as corresponding to $Sp(2,0), Sp(0,2), Sp(1,1)$ and the split
group $Sp(4,\R)$. See Example \ref{ex:equalrank}.
Write \eqref{e:XcupXi} as 
\begin{equation}
\X=\X_{2,0}\cup \X_{0,2}\cup \X_{1,1}\cup \X_s.
\end{equation}

There are $4$ conjugacy classes in $\I_W$, corresponding to the $4$
Cartan subgroups of $Sp(4,\R)$, isomorphic to $S^1\times S^1,
S^1\times \R^\times, \C^\times$ and $\R^\times\times\R^\times$.

Let $\alpha_1,\alpha_2$ be the long positive roots, and
$\beta_1,\beta_2$ the short ones.
Then $\I_W=\{1,s_{\alpha_1},s_{\alpha_2},s_{\beta_1},s_{\beta_2},w_0\}$
where  $w_0=-I$ is the long element.

Here is the output of the {\tt kgb} command of the
{\tt atlas} software for the real form $Sp(4,\R)$:

\begin{verbatim}
 0:  0  0  [n,n]    1   2     6   4
 1:  0  0  [n,n]    0   3     6   5
 2:  0  0  [c,n]    2   0     *   4
 3:  0  0  [c,n]    3   1     *   5
 4:  1  2  [C,r]    8   4     *   *  2
 5:  1  2  [C,r]    9   5     *   *  2
 6:  1  1  [r,C]    6   7     *   *  1
 7:  2  1  [n,C]    7   6    10   *  2,1,2
 8:  2  2  [C,n]    4   9     *  10  1,2,1
 9:  2  2  [C,n]    5   8     *  10  1,2,1
10:  3  3  [r,r]   10  10     *   *  1,2,1,2
\end{verbatim}

Thus $\X_s$ has $11$ elements, labelled by the first column. 
Each row corresponds to an orbit $\caO$, which 
maps to a twisted involution $\tau$. In this case
(since $\gamma=1$) we may view $\tau$ as an involution in $W$; the
last column gives $\tau$ as a product of simple reflections. The
conjugacy class of $\tau$ corresponds to a Cartan subgroup; the number
of this Cartan subgroup, given by the output of the {\tt
cartan} command, is given in column 3. The length of $\caO$, i.e.
$\dim(\caO)-\dim(\caO_{\text{min}})$ where $\caO_{\text{min}}$ is a
minimal orbit, is given in column 2.

The type of each simple root (r=real, c=compact imaginary,
n=noncompact imaginary, C=complex) is given in brackets. Following this
the next
two columns give the cross actions of the simple roots, followed by 
give Cayley transforms by the simple noncompact imaginary roots.

Of course $\X_{2,0}$ and $\X_{0,2}$ are singletons. Finally here is the
output of {\tt kgb} for $Sp(1,1)$.

\begin{verbatim}
0:  0  0  [n,c]   1  0    2  *
1:  0  0  [n,c]   0  1    2  *
2:  1  1  [r,C]   2  3    *  *  1
3:  2  1  [c,C]   3  2    *  *  2,1,2
\end{verbatim}

\newpage
Figure 2 gives a picture of $\X$
\footnote{Thanks to Leslie Saper for
producing this diagram.}. 
As before the vertical columns are the
fibers $\X_\tau$. 
The numbering of the points of $\X$ is from the output of the {\tt
  kgb} commands above. 
The arrows $\rightarrow$ indicate Cayley transforms.

\xyoption{curve}
\begin{figure}[ht]
\begin{equation*}
\xymatrix  @R=1cm @C=1.5cm @!0 {
{\bullet0'''} \ar@{-}[]+<.75cm,.5cm>;[9,0]+<.75cm,-.5cm>
     \save []*++++!CR{\Sympl(2,0)} \restore &
{} & {} \ar@{-}[]+<.75cm,.5cm>;[9,0]+<.75cm,-.5cm> &
{} & {} \ar@{-}[]+<.75cm,.5cm>;[9,0]+<.75cm,-.5cm> & {}
     \save []+<.75cm,.5cm>.{[3,0]+<.75cm,-.5cm>}!C*--\frm{\}},
        !R(.9)*!L!/d .9pt/{\longrightarrow z=I} \restore \\
{\bullet0''}
     \save []*++++!CR{\Sympl(0,2)} \restore & {} & {} & {} & {}
& {} \\
{\bullet0'} \ar[1,1]^(.3){1}
    \save []+<-.75cm,.5cm>.{[1,0]+<-.75cm,-.5cm>}!C*---\frm{\{}
       *!CR{\Sympl(1,1)} \restore
& {} &{} & {} & {} & {} \\
{\bullet1'} \ar[0,1]^(.4){1}
  \ar@{-}[]+<-.75cm,-.5cm>;[0,5]+<.75cm,-.5cm> &
{\bullet2'} & {\bullet3'} & {} & {} & {} \\
{\bullet0} \ar[3,1]^(.6){1} \ar[3,4]^(.35){2}
    \save []+<-.75cm,.5cm>.{[3,0]+<-.75cm,-.5cm>}!C*---\frm{\{}
       *!CR{\Sympl(4,\mathbb R)} \restore
& {} & {} & {} & {} & {}
    \save []+<.75cm,.5cm>.{[3,0]+<.75cm,-.5cm>}!C*--\frm{\}},
       !R(.9)*!L!/d .5pt/{\longrightarrow z=-I} \restore \\
{\bullet1} \ar[2,1]_(.3)*-<2pt>{\scriptstyle 1} \ar[1,4]^(.25){2}
  & {} & {} & {} & {} & {} \\
{\bullet2} \ar[1,4]^(.3){2}  & {} & {} & {\bullet9} & {\bullet5} \\
{\bullet3} \ar[-1,4]^(.57){2}  \ar@{=}[]+<-.75cm,-.5cm>;[0,5]+<.75cm,-.5cm>
  & {\bullet6} &
{\bullet7} \ar@/^1pc/[0,3]^(.33){1} &
{\bullet8} \ar@/^/[0,2]^(.4)*-<2pt>{\scriptstyle 2} &
{\bullet4} \ar[0,1]_(.4){2}
  & {\bullet10} \\
*++++!{1}  \save []*++++!!R{\I_W:} \restore & *!{s_{\alpha_1}} &
*!{s_{\alpha_2}} &
*!{s_{\beta_1}} & *!{s_{\beta_2}} & *!{-1} \\
{} \save []*!{(S^1)^2} \restore \save []*+++++!!R{\text{Cartan:}} \restore &
{}  \save
[].[0,1]!C*!{\mathbb C^\times} \restore & {} & {} \save
[].[0,1]!C*!{S^1 \times \mathbb R^\times} \restore & {} & *!{(\R^\times)^2}
   }
\end{equation*}
\caption{
$\X$ and Cayley transforms for $\Sympl(4)$
}
\end{figure}
See \cite{adams_snowbird} for more detail about the representation
theory of real forms of $Sp(4,\C)$.
\end{exampleplain}

\sec{The Tits group and the algorithmic enumeration of parameters}
\label{s:tits}

The combinatorial enumeration of $K\bs G/B$ is in terms of the {\it Tits
group} $\tW$, introduced by Jacques Tits in \cite{tits_group} with the name
{\it extended Coxeter group}.

We begin by fixing $(G,\gamma)$ and a choice of 
of splitting datum
$S=(H,B,\{X_\alpha\})$ (cf.~Section \ref{s:groups}).
For each simple root $\alpha$
let $\phi_\alpha:SL(2)\rightarrow G$ be defined by
$\phi_\alpha(\diag(z,\frac1z))=\ch\alpha(z)$ and $d\phi_\alpha
\begin{pmatrix}0&1\\0&0\end{pmatrix}=X_\alpha$. Let
$\sigma_\alpha=\phi_\alpha\begin{pmatrix}0&1\\-1&0\end{pmatrix}$. 
This is consistent with the definition of $\sigma_\alpha$ in Section
\ref{s:cayley}.

\begin{definition}
\label{d:tits}
The Tits group $\tW$ is the subgroup of $N$ generated by
$\{\sigma_\alpha\}$ for $\alpha$ simple.
\end{definition}

For each simple root $\alpha$ let
$m_\alpha=\sigma_\alpha^2=\ch\alpha(-1)$. 
Let $H_0$ be the subgroup of $H$ generated by the elements $m_\alpha$.

\begin{theorem}[Tits \cite{tits_group}]
\label{t:tits}
\hfil

\noindent (1) The kernel of the natural map $\tW\rightarrow W$ is
$H_0$,

\noindent (2) The elements $\sigma_\alpha$ satisfy the braid relations,

\noindent (3) There is a canonical lifting of $W$ to a subset of
$\tW$: take a reduced expression $w=s_{\alpha_1}\dots s_{\alpha_n}$,
and let $\tilde w=\sigma_{\alpha_1}\dots \sigma_{\alpha_n}$.
\end{theorem}

It is a remarkable fact that the computations needed for the
enumeration of the $K$-orbits on $G/B$ can be carried out in the Tits
group. This is due to the fact that Cayley transforms and cross
actions can be described entirely in terms of the $\sigma_\alpha$.

\begin{lemma}
\label{l:normalizes}
Given $x\in \X$, there exists a pre-image $\xi$ of $x$ in $\tX$ 
such that $\xi$ normalizes $\tW$. 
\end{lemma}

\begin{proof}
Recall (before Lemma \ref{l:pX})  we write $\delta\in\I_W$ for the image of
the distinguished element  $\delta$ of $\X$.
Using the fact that $p(x)=w\delta$ for some $w\in W$,
it is easy to reduce to the case $x\in\X_{\delta}$.
Let $\theta=\int(\delta)\in \Aut(H)$.

Since $\delta$ is an automorphism of the based root datum used to
define $\tW$, it follows easily that 
$\delta\sigma_\alpha\delta\inv=\sigma_{\theta(\alpha)}$ for all roots
$\alpha$, and therefore $\delta$ normalizes $\tW$.
Choose $\xi\in\tX$ mapping to $x$, so $\xi=h\delta$ for some $h\in H$;
we have $h\theta(h)=\xi^2\in Z(G)$.

Recall (Remark \ref{r:H2})  $H=T_\tau A_\tau$.
We may replace $\xi$ with $g\xi g\inv=g\theta(g\inv)h\delta$ for any $g\in
H$. The map $g\mapsto g\theta(g\inv)$ has image
$A_\tau$, so we may assume $h\in T_\tau$.

We have
$\xi\sigma_\alpha\xi\inv=(h\delta)\sigma_\alpha(h\delta)\inv=h\sigma_{\beta}(h\inv)\sigma_\beta$
where $\beta=\theta(\alpha)$. 
Now $h\theta(h)=h^2\in
Z(G)$, so $\beta(h)=\pm1$. It follows easily
that $h\sigma_\beta(h\inv)$ is equal to $1$ or $m_\beta$, and is therefore
contained in $\tW$.
\end{proof}

Fix $x\in\X$, and choose a pre-image $\xi\in\tX$.
We now describe an  algorithm for computing $\X[x]\simeq
K_\xi\bs G/B$
(cf.~\eqref{e:Xxi}).
By Proposition \ref{p:recapitulation}(2)
we may assume $x\in\X_{\delta}$, and by Lemma
\ref{l:normalizes} that  $\xi$ normalizes $\tW$.
For $\tw\in \tW$ let $\theta(\tw)=\xi\tw\xi\inv$.

We will maintain an intermediate first-in-first-out list of pairs
$(\tau,\tw)$ where $\tau\in\I_W$, $\tw\in\tW$, and $\tw\xi\in\tX_\tau$.
We also maintain a store of  elements $\tW$.
If $\tw$ is in the store then the image of $\tw\xi$ in $\X$ is
contained in $\X[x]$; denote this element $\tw x$.
Initialize the list with $(\delta,1)$.

For each $(\tau,\tw)$ occuring we will keep a record of of $\gr_{\tw x}$;
we assume we are given $\gr_x$.
Let $M_\tau$ be the subgroup of $H_0$ generated by
$\{m_\alpha\,|\,\tau(\alpha)=\alpha\}$.
For each $\tau$ occuring we will compute $M_\tau\cap A_\tau$ 
(cf.~\eqref{e:H-tau}). 
We assume we are 
are given $M_{\delta}\cap A_{\delta}$.

If the list is non-empty, remove from it the first element
$(\tau,\tw)$.

First add $\tw$ to the store.  At the first step ($\tw=1$), record
$\gr_x$. Otherwise $\tw$ is either
of the form $\sigma_\alpha \tilde u\theta(\sigma_\alpha)\inv$ or
$\sigma_\alpha\tilde u$ for  some $\tilde u$ already in the store (see
below).  Compute $\gr_{\tw x}$ by Lemmas
\ref{l:grading} and \ref{l:cross}, and record this information.

Next, compute the orbit of $\tw x$ under the cross action of 
$W_{i,\tau}$ as follows.
Suppose $\tau(\alpha)=\alpha$ and $\gr_{\tw x}(\alpha)=1$.
Choose a representative $\sigma_\alpha\in\tW$ of $s_\alpha$ ($\alpha$ is not
necessarily simple).
By \eqref{e:xisigmaalpha}
$\sigma_\alpha\tw\xi\sigma_\alpha\inv=\sigma_\alpha^2\tw\xi$. 
Repeating this we obtain a collection of elements of the form $\{t\tw\xi\,|\,t\in
S\}$ where $S$ is a subset of $M_\tau$.
The $W_{i,\tau}$-orbit of $x$ is in bijection with
the image of  $S$ in  the group $M_\tau/M_\tau\cap A_\tau$.
Choose  representatives of this set.

For each such representative $t$ let $\wt u=t\tw$ and add  $\wt u$ to the store.
Record $\gr_{\wt u x}$ by writing $\wt ux=w\times x$ for $w\in
W_{i,\tau}$ and using Lemma \ref{l:cross}.
For each simple root $\alpha$ satisfying $\tau(\alpha)=\alpha$ and $\gr_{\tilde ux}(\alpha)=1$,
see if any element of the form  $(s_\alpha\tau,*)$ is on the list. If not, add
$(s_\alpha\tau,\sigma_\alpha\tilde u)$ to the list.
Compute the set  $M_{s_\alpha\tau}\cap A_{s_\alpha\tau}=\langle
M_\tau\cap A_\tau,m_\alpha\rangle$
(cf.~Lemma \ref{l:cayley3}).

Next, for each
simple root $\alpha$, check if $(s_\alpha\tau s_\alpha,*)$ is on the
list. If not add $(s_\alpha\tau s_\alpha,\sigma_\alpha\tw
\theta(\sigma_\alpha)\inv)$ to the list  and compute $M_{s_\alpha\tau
  s_\alpha}\cap A_{s_\alpha\tau s_\alpha}=s_\alpha\cdot M_\tau\cap A_\tau$.

Continue until the list is empty, at which point
$\X[x]$ is the set of elements $\tw x$ for $\tw$ in the store.

\begin{remarkplain}
Fix $x\in\X_{\delta}$ and choose $\xi$ as in the Lemma.  The
argument shows that
the following information suffices
to compute $\X[x]$:
$\tW$,
the involution $\theta=\int(\xi)$ of $\tW$, the grading $\gr_x$
of the $\delta$-imaginary roots, and
the two-group $M_{\delta}\cap A_{\delta}\subset \tW$.

Note that if (the derived group of) $G$ is simply
connected then the grading $\gr_{x}$ is determined by $\tW$ and $\theta$ 
(cf.~Remark \ref{r:grading}).
\end{remarkplain}

\begin{remarkplain}
We may also compute the entire space $\X$ by similar (in fact somewhat
easier) methods,
using Proposition \ref{p:fibers} to describe the fiber
$\X_{\delta}$, and Lemma \ref{l:cayley3} to compute Cayley
transforms.
See Remark \ref{r:computeX}.
We omit the details.
\end{remarkplain}

Carrying out the algorithm uses computations in the Tits group, and 
we briefly sketch how to carry these out.
According to
Theorem \ref{t:tits}, each element of the Tits group can be written uniquely
as $\wt wh$, with $\wt w$ the canonical representative of $w\in W$
and $h\in H_0$.

Fix $w\in W$ and  a simple root $\alpha$, with corresponding
reflection $s=s_\alpha$. We first compute $\wt w\wt s$.
If $l(ws)> l(w)$, then $\wt w\wt s$ is the canonical representative of $ws$, and
we have $\wt w\wt s=\wt{ws}$.
Otherwise $w=vs$ for $v\in W$ with $l(v)=l(w)-1$. In this case
$\wt w \wt s = \wt v \wt s^2=\wt vm_\alpha$.
There is a similar formula for $\wt s\wt w$.

Now for $h\in H_0$ we have
\begin{equation}
(\wt w h)\wt s=(\wt w\wt s)h^s,\quad \wt s(\wt wh)=(\wt s\wt w)h
\end{equation}
with $\wt w\wt s$ or $\wt s \wt w$ computed as above.

In addition, for $h_i\in H_0$ we have
\begin{equation}
h_1(\wt wh_2)h_3=\wt w(h_1^wh_2h_3).
\end{equation}

Therefore multiplication in $\wt W$ can be computed from multiplication in $W$,
multiplication in $H_0$, and the action of $W$ on $H_0$.  
\bibliographystyle{plain}
\def\cprime{$'$} \def\cftil#1{\ifmmode\setbox7\hbox{$\accent"5E#1$}\else
  \setbox7\hbox{\accent"5E#1}\penalty 10000\relax\fi\raise 1\ht7
  \hbox{\lower1.15ex\hbox to 1\wd7{\hss\accent"7E\hss}}\penalty 10000
  \hskip-1\wd7\penalty 10000\box7}

\enddocument
\end